\documentclass[a4paper,twoside,10pt,fleqn]{article}
\usepackage{expl3}
\usepackage{amssymb,amsmath,amsthm}
\usepackage[all]{xy}
\usepackage[nottoc, notlof, notlot]{tocbibind}

\usepackage{tikz}
\usetikzlibrary{arrows}

\makeatletter
\def\blfootnote{\gdef\@thefnmark{}\@footnotetext}
\makeatother
\title{On a motivic invariant of the arc-analytic equivalence}
\newcommand\shorttitle{On a motivic invariant of the arc-analytic equivalence}
\author{Jean-Baptiste Campesato\footnote{\parbox[t][2em][s]{\textwidth}{Univ. Nice Sophia Antipolis, CNRS,  LJAD, UMR 7351, 06100 Nice, France. \newline E-mail address: \url{Jean-Baptiste.CAMPESATO@unice.fr}}}}
\date{April 15, 2016}

\usepackage{ifxetex}
\ifxetex
    \usepackage{polyglossia}
    \setdefaultlanguage[variant=us]{english}
    \usepackage{fontspec}
    \usepackage{unicode-math}
\else
    \usepackage[english]{babel}
    \usepackage{mathrsfs}
    \usepackage{fouriernc}
    \DeclareSymbolFont{calletters}{OMS}{cmsy}{m}{n}
    \DeclareSymbolFontAlphabet{\mathcal}{calletters}
    
    \newcommand{\M}{\mathcal M}
    \usepackage{tgpagella}
    \usepackage[T1]{fontenc}
    \usepackage[utf8x]{inputenc}
\fi

\usepackage{hyperref}
\usepackage{color}
\definecolor{darkred}{rgb}{.5,0,0}
\definecolor{darkgreen}{rgb}{0,.5,0}
\definecolor{darkblue}{rgb}{0,0,.5}
\hypersetup{
    pdfencoding=unicode,
    colorlinks=true,
    linkcolor=darkred,
    citecolor=darkgreen,
    urlcolor=darkblue,
    bookmarks=false,
    pdftoolbar=true,
    pdfmenubar=true,
    pdffitwindow=true,
    pdfstartview={FitH},
    pdfcreator={},
    pdfproducer={},
    pdfnewwindow=true,
}
\makeatletter
    \hypersetup{pdftitle={\shorttitle},pdfauthor={Jean-Baptiste Campesato}}
\makeatother

\usepackage[inner=3cm,outer=2cm,top=1.5cm,bottom=1.5cm,includeheadfoot,headsep=.5cm]{geometry}
\usepackage{fancyhdr}
\renewcommand{\thepage}{\arabic{page}}
\fancypagestyle{plain}{
\fancyhf{}

}

\def\footindent{2em}
\makeatletter
\renewcommand\@makefntext[1]{\leftskip=\footindent\hskip-\footindent\@makefnmark#1}
\makeatother
\usepackage{perpage}
\MakePerPage{footnote}
\makeatletter
\ifxetex
\newcommand*{\fnsymbolsingle}[1]{\ensuremath{\ifcase#1\or\star\or\dagger\or\ddagger\or\textsection\or\|\or\textparagraph\or\else\@ctrerr\fi}}
\else
\newcommand*{\fnsymbolsingle}[1]{\ensuremath{\ifcase#1\or\star\or\dagger\or\ddagger\or\mathsection\or\|\or\mathparagraph\or\else\@ctrerr\fi}}
\fi
\makeatother
\usepackage{alphalph}
\newalphalph{\fnsymbolmult}[mult]{\fnsymbolsingle}{}

\makeatletter
\renewcommand{\maketitle}{
  \newpage
  \null
  \thispagestyle{plain}
  
  \begin{center}
    {\LARGE \@title \par}
    \vskip 1.5em
    {\large
      \lineskip .5em
        \@author
      \par}
    \vskip 1em
    {\large \@date}
  \end{center}
  \par
  \vskip 1.5em}
\makeatother

\theoremstyle{plain}
\newtheorem{thm}{Theorem}[section]
\newtheorem{prop}[thm]{Proposition}
\newtheorem{cor}[thm]{Corollary}
\newtheorem{lemma}[thm]{Lemma}
\theoremstyle{definition}
\newtheorem{defn}[thm]{Definition}
\newtheorem{defns}[thm]{Definitions}
\newtheorem{eg}[thm]{Example}
\newtheorem{rem}[thm]{Remark}

\newtheorem{notation}[thm]{Notation}

\usepackage{enumitem}
\setlist[itemize]{labelindent=.6em, itemindent=1em, leftmargin=!, label=\textbullet}

\usepackage{ifmtarg}
\makeatletter
  \newcommand{\TODO}[1]{\@ifmtarg{#1}{\emph{\textbf{TODO}}~}{\emph{\textbf{TODO:}~#1~}}}
\makeatother

\newcommand{\Var}{\operatorname{Var}}
\newcommand{\ord}{\operatorname{ord}}
\newcommand{\ac}{\operatorname{ac}}
\newcommand{\pring}[1]{\ensuremath{#1^\bullet}}

\newcommand{\quotient}[2]{\left.\raisebox{.05em}{$#1$}\middle/\raisebox{-.05em}{$#2$}\right.}
\renewcommand{\d}{\mathrm{d}}
\newcommand{\X}{\mathfrak X}
\newcommand{\Y}{\mathfrak Y}

\newcommand{\Jac}{\operatorname{Jac}}
\newcommand{\mon}{\mathrm{mon}}

\newcommand{\pr}{\operatorname{pr}}
\newcommand{\Sf}{\mathscr S}
\newcommand{\AS}{\mathcal{AS}}
\newcommand{\mult}{\operatorname{mult}}
\newcommand{\Conv}{\operatorname{Conv}}
\newcommand{\id}{\operatorname{id}}
\renewcommand{\M}{\mathcal M}

\newcommand{\bigast}{\mathop{\scalebox{1.5}{\raisebox{-0.2ex}{$*$}}}}

\makeatletter
\newcommand{\vast}{\bBigg@{4}}
\newcommand{\Vast}{\bBigg@{5}}
\makeatother

\usepackage{graphicx}
\def\acts{\ensuremath{\rotatebox[origin=c]{-90}{$\circlearrowright$}}}

\begin{document}
\maketitle

\begin{abstract}
To a Nash function germ, we associate a zeta function similar to the one introduced by J. Denef and F. Loeser. Our zeta function is a formal power series with coefficients in the Grothendieck ring $\mathcal{M}$ of $\mathcal{AS}$-sets up to $\mathbb{R}^*$-equivariant $\mathcal{AS}$-bijections over $\mathbb{R}^*$, an analog of the Grothendieck ring constructed by G. Guibert, F. Loeser and M. Merle. This zeta function generalizes the previous construction of G. Fichou but thanks to its richer structure it allows us to get a convolution formula and a Thom--Sebastiani type formula.

We show that our zeta function is an invariant of the arc-analytic equivalence, a version of the blow-Nash equivalence of G. Fichou. The convolution formula allows us to obtain a partial classification of Brieskorn polynomials up to arc-analytic equivalence by showing that the exponents are arc-analytic invariants.
\end{abstract}

\tableofcontents

\section{Introduction}
In the context of motivic integration, J. Denef and F. Loeser \cite{DL98} \cite{DL02} associate a formal power series to a function $f:X\rightarrow\mathbb A^1$ defined on a non-singular (scheme theoric) algebraic variety. This power series, called the motivic zeta function of $f$, comes in various forms by modifying the ring where its coefficients lie. For instance J. Denef and F. Loeser work with the classical Grothendieck ring of algebraic varieties in order to define the naive motivic zeta function or with an equivariant Grothendieck ring which encodes actions of the roots of unity in order to define the equivariant motivic zeta function. This equivariant structure allows them to get a convolution formula \cite{DL99-TS} which computes a modified equivariant zeta function of $f\oplus g(x,y)=f(x)+g(y)$ by applying coefficientwise a convolution product to the modified equivariant zeta functions of $f$ and $g$.

The key lemma for the motivic change of variables formula \cite{DL99-Germs} ensures that their zeta functions are rational. Thus they admit a limit at infinity whose multiplication by $-1$ is called motivic Milnor fibers since their known realizations coincide with the ones of the classical Milnor fiber. The convolution formula induces a Thom--Sebastiani type formula for these motivic Milnor fibers.

Similarly S. Koike and A. Parusiński \cite{KP03} associate to a real analytic function germ a formal power series with coefficients in $\mathbb Z$  using the Euler characteristic with compact support. This way they define a naive zeta function and two zeta functions with sign (a positive one and a negative one) which play the role of the equivariant zeta function. Particularly these zeta functions with sign admit formulas similar to the ones of Denef--Loeser such as a convolution formula. Thanks to an adaptation to the real analytic case of the key lemma for the motivic change of variables formula, it turns out that Koike--Parusiński zeta functions are invariants of the blow-analytic equivalence of T.-C. Kuo \cite{Kuo85} for real analytic germs.

G. Fichou \cite{Fic05} brings a richer structure by using the virtual Poincaré polynomial \cite{MP03,Fic05,MP11} for $\AS$-sets instead of the Euler characteristic with compact support. This way he defines a naive zeta function and two zeta functions with sign. In order to get a rationality formula, he has to restrict to Nash functions. For this reason he introduces a semialgebraic version of the blow-analytic equivalence for Nash germs, called the blow-Nash equivalence. As it is not known if this relation is an equivalence relation, he introduces a more general notion of blow-Nash equivalence \cite{Fic05-bis} in terms of Nash modifications which is an equivalence relation but two Nash germs which are blow-Nash equivalent in this sense have to satisfy an additional condition to ensure they have the same zeta functions. Recently one uses a third definition of the blow-Nash equivalence, that is midway between the both previous ones \cite{FF,Fic06,Fic08} by adding the above cited additional condition, but it was not obvious originally whether it is an equivalence relation. We show it in Corollary \ref{cor:BNisER} by using the notion of arc-analytic equivalence that we introduce in Section \ref{sect:bne}.

G. Guibert, F. Loeser and M. Merle \cite{GLM} introduce an equivariant Grothendieck ring for actions of the multiplicative group on (scheme theoric) algebraic varieties over some fixed algebraic variety (this Grothendieck ring is equivalent to the one of Denef--Loeser for actions of the roots of unity). In this paper we first adapt this framework to $\AS$-sets up to $\mathbb R^*$-equivariant $\AS$-bijections over $\mathbb R^*$.

This will allow us to define a local zeta function similar to the one of Fichou but with additional structures. After having highlighted the links between our zeta function to the ones of Koike--Parusiński and Fichou, we show that our zeta function is also rational.

The main part of this work consists in proving that our additional structures permit to define a convolution formula allowing us to compute the zeta function of $f\oplus g$. 
More precisely, we construct a new convolution product on our Grothendieck ring which is compatible with a modified zeta function in the following sense. If we apply coefficientwise this convolution product to the modified zeta functions of $f$ and $g$, we get the modified zeta function of $f\oplus g$.

We may notice that this modified zeta function has already appeared in \cite{DL99-TS} and \cite{KP03} in their respective settings. In a way similar to \cite{KP03} we prove that our modified zeta function contains the same information as our zeta function.

The next part of this paper is devoted to the study of the behavior of our zeta function under the blow-Nash equivalence. To this purpose we introduce a new relation, the arc-analytic equivalence, we show that it is an equivalence relation and that allows us to avoid using Nash modifications. Moreover, we show that the arc-analytic equivalence coincides with the blow-Nash equivalence in the sense of \cite{FF,Fic06,Fic08}. Our zeta function and the ones of Fichou are invariants of this relation. 

Finally, the convolution formula allows us to prove that the exponents of Brieskorn polynomials are invariants of the arc-analytic equivalence. \\

\noindent\textbf{Acknowledgements.} I would like to express my gratitude to my thesis advisor Adam Parusiński for his support and helpful discussions during the preparation of this work.

\section{Geometric framework}
K. Kurdyka \cite{Kur88} introduced semialgebraic arc-symmetric subsets of $\mathbb R^d$ which are semialgebraic subsets of $\mathbb R^d$ such that given a real analytic arc on $\mathbb R^d$ either this arc is entirely included in the subset or it meets it at isolated points only. We are going to work with $\AS$-sets, a slightly different notion introduced by A. Parusiński \cite{Par04}.
\begin{defn}[\cite{Par04}]
We say that a semialgebraic subset $A\subset\mathbb P_\mathbb R^n$ is an \emph{$\mathcal{AS}$-set} if for every real analytic arc $\gamma:(-1,1)\rightarrow\mathbb P_\mathbb R^n$ such that $\gamma((-1,0))\subset A$ there exists $\varepsilon>0$ such that $\gamma((0,\varepsilon))\subset A$.
\end{defn}

\begin{defn}
By an $\AS$-map we mean a map between $\AS$-sets whose graph is $\AS$ and by an $\AS$-isomorphism we mean a bijection between $\AS$-sets whose graph is $\AS$.
\end{defn}

\begin{rem}[{\cite[Remark 4.2]{Par04}}]
The family $\AS$ is the boolean algebra spanned by semialgebraic arc-symmetric subsets of $\mathbb P_\mathbb R^n$ (in the sense of \cite[Définition 1.1]{Kur88}). Moreover closed (for the Euclidean topology) $\AS$-sets are exactly the semialgebraic arc-symmetric subsets of $\mathbb P_\mathbb R^n$.
\end{rem}

The following proposition results from the proof of \cite[Theorem 2.5]{Par04}. It is an $\AS$ version of \cite[Théorème 1.4]{Kur88}.
\begin{prop}
The closed $\AS$-subsets of $\mathbb P_\mathbb R^n$ are exactly the closed sets of a noetherian topology on $\mathbb P_\mathbb R^n$.
\end{prop}

\begin{defn}
Let $Y,X,F$ be three $\AS$-sets and $p:Y\rightarrow X$ be an $\AS$-map. We say that $p$ is a locally trivial $\AS$-fibration with fibre $F$ if every $x\in X$ admits an open $\AS$ neighborhood $U$ such that $\varphi:p^{-1}(U)\rightarrow U\times F$ is an $\AS$-isomorphism such that the following diagram commutes $$\xymatrix{p^{-1}(U) \ar[rr]^\varphi_\simeq \ar[dr]_p & & U\times F \ar[ld]^{\pr_U} \\ & U &}$$ 
\end{defn}

The following proposition is a direct consequence of the noetherianity of the $\AS$ topology.
\begin{prop}\label{prop:ptfAS}
A locally trivial $\AS$-fibration $p:Y\rightarrow X$ with fibre $F$ is a piecewise trivial fibration with fiber $F$, i.e. we may break $X$ into finitely many $\AS$-sets $X=\bigsqcup_{\alpha=1}^kX_\alpha$ such that $p^{-1}(X_\alpha)$ is $\AS$-isomorphic to $X_\alpha\times F$.
\end{prop}

Nash maps and Nash manifolds were introduced by J. Nash in \cite{Nas52} considering real analytic functions satisfying non-trivial polynomial equations. M. Artin and B. Mazur gave a new description of these objects in \cite{AM65} in terms of maps which can be lifted to polynomial maps on non-singular irreducible real algebraic sets.

The benefits of Nash functions is that they share good algebraic properties with polynomial functions and good geometric properties with real analytic geometry.

\begin{defn}[Nash functions and Nash maps {\cite{Nas52}\cite{BCR}}]
Let $U\subset\mathbb R^d$ be an open semialgebraic subset. A function $f:U\rightarrow\mathbb R$ is said to be Nash if it satisfies one of the two following equivalent conditions:
\begin{enumerate}[nosep]
\item $f$ is semialgebraic and $C^\infty$.
\item $f$ is analytic and satisfies a nontrivial polynomial equation.
\end{enumerate}
A map $f:U\rightarrow\mathbb R^n$ is Nash if its coordinate functions are Nash.
\end{defn}

\begin{rem}
Obviously, the zero locus of a Nash function is an $\AS$-set.
\end{rem}

\begin{defn}
A Nash submanifold of dimension $d$ is a semialgebraic subset $M$ of $\mathbb R^n$ such that for every $x\in M$ there exist an open semialgebraic neighborhood $U$ of $x$ in $\mathbb R^n$, an open semialgebraic neighborhood $V$ of $0$ in $\mathbb R^n$ and a Nash isomorphism $\varphi:U\rightarrow V$ such that $\varphi(x)=0$ et $\varphi(M\cap U)=(\mathbb R^d\times\{0\})\cap V$.
\end{defn}

\begin{eg}[{\cite[Proposition 3.3.11]{BCR}}]
Let $V\subset\mathbb R^d$ be a non-singular\footnote{Every point of $V$ are non-singular in the same dimension, see \cite[Definition 3.3.4]{BCR}.} algebraic set. Then, by the Jacobian criterion \cite[Proposition 3.3.8]{BCR} and the Nash inverse theorem \cite[Proposition 2.9.7]{BCR}, $V$ is a Nash submanifold of $\mathbb R^d$.
\end{eg}

\begin{prop}\label{prop:NashResol}
Let $f:V\rightarrow\mathbb R$ be a Nash function defined on an algebraic set $V$. There exists $\sigma:\tilde V\rightarrow V$ a finite sequence of (algebraic) blowings-up with non-singular centers such that $\tilde V$ is non-singular and $f\circ\sigma:\tilde V\rightarrow\mathbb R$ has only normal crossings\footnote{i.e.every $x\in\tilde V$ admits an open $\AS$-neighborhood such that $f\circ\sigma$ equals a monomial times a nowhere vanishing function on this neighborhood.}.
\end{prop}
\begin{proof}
Let $G(x,y)=\sum_{i=0}^pG_i(x)y^{p-i}$ be a non-trivial polynomial such that $G(x,f(x))=0$. We may assume that $G_p\neq0$. Then $f$ divides $G_p$ seen as a Nash function.

There exists a finite sequence of algebraic blowings-up with non-singular centers such that $G_p\circ\sigma$ is monomial, thus $f\circ\sigma$ is monomial.
\end{proof}

\section{A Grothendieck ring}
In this section we adapt the Grothendieck ring introduced in \cite{GLM} to our framework over $\mathbb R$. The classical Grothendieck ring of $\AS$-sets doesn't allow one to get a convolution formula similar to the one of Section \ref{sect:TS} and the lack of real roots of unity prevents us from using an equivariant Grothendieck ring similar to the one of \cite{DL98} or \cite{DL01} as in \cite{DL99-TS}. \\

\begin{defn}
We denote by $K_0(\AS)$ the free abelian group spanned by symbols $[X]$ where $X\in\AS$ with the relations
\begin{enumerate}[ref=\thesubsection.(\arabic*), label=(\arabic*)]
\item If there is a bijection whose graph is $\AS$ between $X$ and $Y$ then $[X]=[Y]$
\item If $Y\subset X$ is a closed $\AS$-subset then $[X\setminus Y]+[Y]=[X]$
\end{enumerate}
Moreover we have a ring structure induced by the cartesian product $[X\times Y]=[X][Y]$. The unit of the sum is $0=[\varnothing]$ and the one of the product is $1=[\mathrm{pt}]$.
\end{defn}

\begin{rem}
The group $K_0(\AS)$ is well-defined since $\AS$ is a set.
\end{rem}

\begin{rem}
If $A,B\in\AS$ then $[A\sqcup B]=[A]+[B]$. We may prove this observation using that an $\AS$-set may be written as a finite disjoint union of locally closed $\AS$-sets.
\end{rem}

We denote by $\mathbb L_{\AS}=[\mathbb R]$ the class of the affine line and by $\M_{\AS}=K_0(\AS)\left[\mathbb L_{\AS}^{-1}\right]$ the localization by the class of the affine line. \\

We define the category $\Var^n_\mon$ whose objects are $\AS$-sets $X$ endowed with an $\AS$ action of $\mathbb R^*$ (i.e. the graph of the map $\mathbb R^*\times X\rightarrow X$ defined by $(\lambda,x)\mapsto\lambda\cdot x$ is $\AS$) together with an $\AS$-map $\varphi_X:X\rightarrow\mathbb R^*$ such that $\varphi_X(\lambda\cdot x)=\lambda^n\varphi_X(x)$. \\%In order to be able to work locally, we request the action to be good, i.e. each orbit must be included in an affine open set. \\
A morphism is an equivariant $\AS$ map $f:X\rightarrow Y$ over $\mathbb R^*$, i.e. $f(\lambda\cdot x)=\lambda\cdot f(x)$ and the following diagram commutes $$\xymatrix{X \ar[rr]^f \ar[rd]_{\varphi_X} & & Y \ar[dl]^{\varphi_Y} \\ & \mathbb R^* & }$$

\begin{defn}\label{defn:K0n}
We denote by $K_0^n$ the free abelian group spanned by the symbols $\left[\varphi_X:\mathbb R^*\acts X\rightarrow\mathbb R^*\right]$ where $\varphi_X:\mathbb R^*\acts X\rightarrow\mathbb R^*\in\Var^n_\mon$ with the relations
\begin{enumerate}[ref=\ref{defn:K0n}.(\arabic*), label=(\arabic*)]
\item If there is $f:X\rightarrow Y$ an equivariant bijection over $\mathbb R^*$ whose graph is $\AS$, i.e. $$\xymatrix{X \ar[rr]^f_\simeq \ar[rd]_{\varphi_X} & & Y \ar[dl]^{\varphi_Y} \\ & \mathbb R^* &}$$ then $\left[\varphi_X:\mathbb R^*\acts X\rightarrow\mathbb R^*\right]=\left[\varphi_Y:\mathbb R^*\acts Y\rightarrow\mathbb R^*\right]$
\item If $Y$ is a closed $\AS$-subset of $X$ invariant by the action of $\mathbb R^*$ then $$\left[\varphi_X:\mathbb R^*\acts X\rightarrow\mathbb R^*\right]=\left[\varphi_{X|X\setminus Y}:\mathbb R^*\acts X\setminus Y\rightarrow\mathbb R^*\right]+\left[\varphi_{X|Y}:\mathbb R^*\acts Y\rightarrow\mathbb R^*\right]$$
\item \label{item:relationaffine} Let $\varphi:\mathbb R^*\acts_{\tau}Y\rightarrow\mathbb R^*\in\Var^n_\mon$ and $\psi=\varphi\circ\pr_Y:Y\times\mathbb R^m\rightarrow\mathbb R^*$. Let $\sigma$ and $\sigma'$ be two actions on $Y\times\mathbb R^m$ that are two liftings of $\tau$ on $Y$, i.e. $\pr_Y(\lambda\cdot_\sigma x)=\pr_Y(\lambda\cdot_{\sigma'}x)=\lambda\cdot_\tau\pr_Y(x)$. Then $\psi:\mathbb R^*\acts_{\sigma}(Y\times\mathbb R^m)\rightarrow\mathbb R^*$ and $\psi:\mathbb R^*\acts_{\sigma'}(Y\times\mathbb R^m)\rightarrow\mathbb R^*$ are in $\Var^n_\mon$ and we add the relation $\left[\psi:\mathbb R^*\acts_{\sigma}(Y\times\mathbb R^m)\rightarrow\mathbb R^*\right]=\left[\psi:\mathbb R^*\acts_{\sigma'}(Y\times\mathbb R^m)\rightarrow\mathbb R^*\right]$.\footnote{This relation will allow us to focus on the angular component, we can move out the other coefficients, e.g. in the rationality formula.}
\end{enumerate}
Moreover, $K_0^n$ has a ring structure given by the fiber product over $\mathbb R^*$ where the action is diagonal. Furthermore the class $\mathbb 1_n=[\id:\mathbb R^*\acts\mathbb R^*\rightarrow \mathbb R^*]$, where $\mathbb R^*$ acts on $\mathbb R^*$ by $\lambda\cdot a=\lambda^na$, is the unit of this product. \\
Finally, the cartesian product induces a structure of $K_0(\AS)$-algebra by $$K_0(\AS)\rightarrow K_0^n,\,[A]\mapsto[A]\cdot\mathbb1_n=[A\times\mathbb R^*\rightarrow\mathbb R^*]$$ where the action of $\mathbb R^*$ on $A$ is trivial. \\
Particularly we set $\mathbb L_n=\mathbb L_{\AS}\mathbb 1_n=[\pr_{\mathbb R^*}:\mathbb R\times\mathbb R^*\rightarrow \mathbb R^*]\in K_0^n$ and $\M^n=K_0^n\left[\mathbb L_n^{-1}\right]$.
\end{defn}

\begin{notation}
We may simply denote $\left[\varphi_X:\mathbb R^*\acts_\sigma X\rightarrow\mathbb R^*\right]$ by $\left[\varphi_X,\sigma\right]$.
\end{notation}

We consider the directed partial order $\prec$ on $\mathbb N\setminus\{0\}$ defined by $n\prec m\Leftrightarrow\exists k\in\mathbb N\setminus\{0\},n=km$. For $n\prec m$, we define the morphism $\theta_{mn}:\Var_\mon^m\rightarrow\Var_\mon^n$ which keeps the same object, the same morphism but which replaces the action by $\lambda\cdot_n x=\lambda^k\cdot_mx$. Thus $(\Var_\mon^n)_{n\ge1}$ is an inductive system and we set $\Var_\mon=\varinjlim\Var_\mon^n$. We define $K_0=\varinjlim K_0^n$ and $\M=\varinjlim\M^n$ in the same way. \\
Thereby $K_0$ has a natural structure of $K_0(\AS)$-algebra and $\M$ has a natural structure of $\M_{\AS}$-algebra. The product unit of $K_0$ is $\mathbb 1=\varinjlim\mathbb 1_n\in K_0$. 
We notice that $\varinjlim\mathbb L_n=\mathbb L_{\AS}\mathbb1\in K_0$ (for the scalar multiplication) and we denote this class by $\mathbb L$. We also notice that $\mathcal M=K_0\left[\mathbb L^{-1}\right]$.

\begin{rem}
$\mathbb L=\left[\pr_{\mathbb R^*}:\mathbb R\times\mathbb R^*\rightarrow\mathbb R^*\right]$
\end{rem}

We have a forgetful morphism of $K_0(\AS)$-modules $\overline{\vphantom{1em}\ \cdot\ }:K_0^n\rightarrow K_0(\AS)$ induced by $$[\varphi_X:\mathbb R^*\acts X\rightarrow\mathbb R^*]\in K_0^n\mapsto [X]\in K_0(\AS)$$ This morphism isn't compatible with the ring structures as shown in the next example. It extends to a morphism of $K_0(\AS)$-modules $\overline{\vphantom{1em}\ \cdot\ }:K_0\rightarrow K_0(\AS)$ and to morphisms of $\M_\AS$-modules $\overline{\vphantom{1em}\ \cdot\ }:\M^n\rightarrow\M_{\AS}$ and $\overline{\vphantom{1em}\ \cdot\ }:\M\rightarrow\M_{\AS}$.

\begin{eg}
$\overline{\mathbb 1}=[\mathbb R^*]=\mathbb L_{\AS}-1\neq 1\in K_0(\AS)$
\end{eg}

\begin{rem}
Let $A\in K_0$ and $n\in\mathbb N$ so that $\frac{A}{\mathbb L^n}\in\M$ then $\overline{\left(\frac{A}{\mathbb L^n}\right)}=\frac{\overline A}{\mathbb L_{\AS}^n}\in\M_{\AS}$.
\end{rem}

\section{The motivic local zeta function}
In \cite{DL98} and \cite{DL02}, J. Denef and F. Loeser introduced and studied a motivic global zeta function and defined the motivic Milnor fiber as a limit of this zeta function. In their framework, the realizations of the motivic Milnor fiber and the classical Milnor fiber coincide for the known additive invariants. \\
In real geometry, a similar work was first initiated by S. Koike and A. Parusiński \cite{KP03} using the Euler characteristic with compact support. They defined a naive motivic local zeta function for real analytic functions. They also introduced a positive and a negative zeta function in order to study the equivariant side. Next, G. Fichou \cite{Fic05} defined similar zeta functions of Nash functions using the virtual Poincaré polynomial. These constructions are used to classify real singularities respectively in terms of blow-analytic \cite{Kuo79,Kuo85} and blow-Nash equivalences. \\
Following J. Denef and F. Loeser, as well as G. Guibert, F. Loeser and M. Merle \cite{GLM}, we introduce a motivic local zeta function for Nash germs with coefficients in $\M$. This way, we obtain a richer zeta function which, in particular, encodes the equivariant aspects. By means of a resolution, we show that this zeta function is rational as all of the above-cited zeta functions. Finally, we shall exhibit the links with the zeta functions of S. Koike and A. Parusiński as well with those of G. Fichou.

\subsection{Definition}
\begin{defns}
For $M$ a Nash manifold, we set $$\mathcal L(M)=\left\{\gamma:(\mathbb R,0)\rightarrow M,\,\gamma\text{ real analytic}\right\}$$ and $$\mathcal L_n(M)=\quotient{\mathcal L(M)}{\sim_n}$$ where $\gamma_1\sim_n\gamma_2\Leftrightarrow\gamma_1\equiv\gamma_2\mod t^{n+1}$ in a local Nash coordinate system around $\gamma_1(0)=\gamma_2(0)$.

We have truncation maps $\pi_n:\mathcal L(M)\rightarrow\mathcal L_n(M)$ and $\pi^m_n:\mathcal L_m(M)\rightarrow\mathcal L_n(M)$ where $m\ge n$. These maps are surjective.

J. Nash first studied truncation of arcs on algebraic varieties in order to study singularities in 1964 \cite{Nas95}. They were then studied by K. Kurdyka, M. Lejeune-Jalabert, A. Nobile, M. Hickel and many others. They are a centerpiece of motivic integration developped by M. Kontsevich and then by J. Denef and F. Loeser.

If $h:M\rightarrow N$ is Nash, then $h_*:\mathcal L(M)\rightarrow\mathcal L(N)$ and $h_{*n}:\mathcal L_n(M)\rightarrow\mathcal L_n(N)$ are well-defined and the following diagram commutes $$\xymatrix{\mathcal L(M) \ar[r]^{h_{*}} \ar@{->>}[d]_{\pi_m} & \mathcal L(N) \ar@{->>}[d]^{\pi_m} \\ \mathcal L_m(M) \ar[r]^{h_{*m}} \ar@{->>}[d]_{\pi^m_n} & \mathcal L_m(N) \ar@{->>}[d]^{\pi^m_n} \\ \mathcal L_n(M) \ar[r]^{h_{*n}} & \mathcal L_n(N)}$$
\end{defns}

We refer the reader to \cite[\S2.4]{jbc1} for the properties of $\mathcal L_n(M)$ and $\mathcal L(M)$. \\

Let $f:(\mathbb R^d,0)\rightarrow(\mathbb R,0)$ be a Nash germ and let $$\X_n(f)=\left\{\gamma\in\mathcal L_n(\mathbb R^d),\,\gamma(0)=0,\,f(\gamma(t))\equiv ct^n\mod t^{n+1},\,c\neq0\right\}\text{ for $n\ge1$}$$

Then $\X_n(f)$ is Zariski-constructible and $\left[\X_n(f)\right]$ is well-defined in $K_0^n$ by the morphism $\varphi:\X_n(f)\rightarrow\mathbb R^*$ with $\varphi(\gamma)=\ac(f\gamma)=c$ and the action of $\mathbb R^*$ given by $\lambda\cdot\gamma(t)=\gamma(\lambda t)$. \\

\begin{defn}
The motivic local zeta function of $f$ is defined by $$Z_f(T)=\sum_{n\ge1}\left[\X_n(f)\right]\mathbb L^{-nd}T^n\in\M[[T]]$$
\end{defn}

\begin{eg}
Let $f_k^\varepsilon=\varepsilon x^k$ where $\varepsilon\in\{\pm1\}$. Then $$Z_{f_k^\varepsilon}(T)=[f_k^\varepsilon:\mathbb R^*\rightarrow\mathbb R^*]\frac{\mathbb L^{-1}T^k}{\mathbb 1-\mathbb L^{-1}T^k}$$
\end{eg}
\begin{proof}
Obviously $$\left[\X_{n}(f_k^\varepsilon)\right]\mathbb L^{-n}=\left\{\begin{array}{ll}[f_k^\varepsilon:\mathbb R^*\rightarrow\mathbb R^*]\mathbb L^{-q}&\text{ if $n=kq$}\\0&\text{ otherwise}\end{array}\right.$$
\end{proof}

\subsection{Link with previously defined motivic real zeta functions}
\subsubsection{Koike--Parusiński zeta functions}
\begin{defn}
We denote by $K_0(SA)$ the free abelian group spanned by symbols $[X]$ where $X$ is semialgebraic with the relations
\begin{enumerate}[ref=\thesubsection.(\arabic*), label=(\arabic*)]
\item If there is a semialgebraic homeomorphism $X\rightarrow Y$ then $[X]=[Y]$.
\item If $Y\subset X$ is closed-semialgebraic then $[X\setminus Y]+[Y]=[X]$.
\end{enumerate}
Moreover we have a ring structure induced by the cartesian product
\begin{enumerate}[resume,ref=\thesubsection.(\arabic*), label=(\arabic*)]
\item $[X\times Y]=[X][Y]$.
\end{enumerate}
\end{defn}

\begin{rem}[\cite{Qua01}]
The Grothendieck ring of semialgebraic sets up to semialgebraic homeomorphisms is isomorphic to $\mathbb Z$ via the Euler characteristic with compact support, thus every additive invariant factorises through the Euler characteristic with compact support.
\end{rem}

\begin{notation}
We set $\mathbb L_{SA}=[\mathbb R]\in K_0(SA)$ and $\M_{SA}=K_0(SA)\left[\mathbb L_{SA}^{-1}\right]$.
\end{notation}

\begin{rem}
$K_0(SA)\simeq\mathbb Z\simeq\M_{SA}$
\end{rem}

\begin{rem}
The cartesian product induces a structure of $K_0(\AS)$-module (resp. $\M_{\AS}$-module) on $K_0(SA)$ (resp. $\M_{SA}$).
\end{rem}

\begin{prop}
The maps
$$F^>:\begin{array}{ccc}\Var_\mon^n&\longrightarrow&SA \\ (X,\sigma,\varphi:X\rightarrow\mathbb R^*)&\longmapsto&\varphi^{-1}(\mathbb R_{>0})\end{array}$$
$$F^<:\begin{array}{ccc}\Var_\mon^n&\longrightarrow&SA \\ (X,\sigma,\varphi:X\rightarrow\mathbb R^*)&\longmapsto&\varphi^{-1}(\mathbb R_{<0})\end{array}$$
induce morphisms of $K_0(\AS)$-modules (resp. $\M_{\AS}$-modules)
$$F^>:K_0\rightarrow K_0(SA)\quad\quad\quad\text{(resp. }F^>:\M\rightarrow \M_{SA}\text{)}$$
$$F^<:K_0\rightarrow K_0(SA)\quad\quad\quad\text{(resp. }F^<:\M\rightarrow \M_{SA}\text{)}$$
\end{prop}

\begin{rem}
These morphisms are not compatible with the ring structures. \\
Particularly, the following computation shows that the unit is not mapped to the unit $$\chi_c(F^{>}(\mathbb 1))=\chi_c(\mathbb R_{>0})=-1\neq1=\chi_c(\mathrm{pt})$$
\end{rem}

\begin{rem}
Given a rational fraction in $\M[[T]]$, we can't directly apply the forgetful morphism or the morphisms $F^>,F^<$ to the coefficients in the numerator and the denominator. We first have to develop it in series.

For example, whereas $\sum_{n\ge1}\mathbb 1T^n=\frac{T}{\mathbb 1-T}\in K_0[[T]]$, we have $$\sum_{n\ge1}\overline{\mathbb 1}T^n=(\mathbb L_{\AS}-1)\frac{T}{1-T}\neq\frac{(\mathbb L_{\AS}-1)T}{(\mathbb L_{\AS}-1)-(\mathbb L_{\AS}-1)T}\in K_0(\AS)[[T]]$$
This phenomenon is due to the fact that these morphisms are not compatible with the ring structures.
\end{rem}

\begin{rem}
Let $A\in K_0$ then $F^{\varepsilon}\left(\frac{A}{\mathbb L^n}\right)=\frac{F^{\varepsilon}(A)}{\mathbb L_{SA}^n}$ where $\varepsilon\in\{<,>\}$.
\end{rem}

\begin{prop}
Given $f:(\mathbb R^d,0)\rightarrow(\mathbb R,0)$ a Nash germ, we recover from $Z_f(T)$ the motivic zeta functions considered by S. Koike and A. Parusiński \cite{KP03} applying the previous morphisms and the Euler characteristic with compact support at each coefficient:
$$Z_f^{\chi_c}(T)=\sum_{n\ge1}\chi_c\left(\overline{\left[\X_n(f)\right]}\right)(-1)^{nd}T^n\in\mathbb Z[[T]]$$
and
$$Z_f^{\chi_c,\varepsilon}(T)=\sum_{n\ge1}\chi_c\left(F^\varepsilon\left(\left[\X_n(f)\right]\right)\right)(-1)^{nd}T^n\in\mathbb Z[[T]]$$ where $\varepsilon=>,<$. \\
\end{prop}

\subsubsection{Fichou zeta functions}
C. McCrory and A. Parusiński \cite{MP03} proved that there exists a unique additive invariant of real algebraic varieties which coincides with the Poincaré polynomial for compact non-singular varieties. This construction relies on the weak factorization theorem \cite{AKMW} in order to describe the Grothendieck ring of real algebraic varieties in terms of blowings-up. Then G. Fichou \cite{Fic05} extended this construction to $\AS$-sets up to Nash isomorphisms. Using an extension theorem of F. Guillén and V. Navarro Aznar \cite{GNA}, C. McCrory and A. Parusiński \cite{MP03} proved the virtual Poincaré polynomial is in fact an invariant of $\AS$-sets up to bijections with $\AS$ graph.

\begin{thm}[The virtual Poincaré polynomial for $\AS$-sets \cite{MP03}\cite{Fic05}\cite{MP11}]
There is a unique map $\beta:\AS\rightarrow\mathbb Z[u]$ which factorises through $K_0(\AS)$ as a ring morphism $\beta:K_0(\AS)\rightarrow\mathbb Z[u]$,
$$\xymatrix{\AS \ar[rr]^\beta \ar[rd] & &\mathbb Z[u] \\ & K_0(\AS) \ar[ru]_{\beta} &}$$
such that
\begin{itemize}
\item If $X\in\AS$ is non-empty, then $\deg\beta(X)=\dim X$ and the leading coefficient is positive.
\item If $X\in\AS$ is compact and non-singular, $\beta(X)=\sum_i\dim H_i(X,\mathbb Z_2)u^i$.
\end{itemize}
\end{thm}

\begin{rem}
We recall the argument of the proof of \cite[Theorem 4.6]{MP11} which explains why the virtual Poincaré polynomial is an invariant of $\AS$-sets up to $\AS$-isomorphism. Let $f:X\rightarrow Y$ be an $\AS$-isomorphism (i.e. a bijection whose graph is $\AS$). First we may break $X$ into a finite decomposition of locally compact $\AS$-sets, $X=\sqcup X_i$. Since $f:X_i\rightarrow f(X_i)$ is semialgebraic we may break $X_i$ into a finite decomposition of semialgebraic sets $X_i=\sqcup X_{ij}$, where $f:X_{ij}\rightarrow f(X_{ij})$ is continuous. As explained in the proof of \cite[Theorem 4.6]{MP11}, we may assume that $X_{ij}\in\AS$ using the $\AS$-closure and the noetherianity of the $\AS$-topology. Now we repeat these arguments to $f^{-1}:f(X_{ij})\rightarrow X_{ij}$ in order to get a finite decomposition $X=\sqcup X_{ijk}$ of $X$ into locally compact $\AS$-sets such that $f:X_{ijk}\rightarrow f(X_{ijk})$ is a homeomorphism whose graph is $\AS$. By \cite[Proposition 4.3]{MP11}, $\beta(X_{ijk})=\beta(f(X_{ijk}))$. We conclude using the additivity of the virtual Poincaré polynomial.
\end{rem}

\begin{prop}
The maps
$$F^+:\begin{array}{ccc}\Var_\mon^n&\longrightarrow&\AS \\ (X,\sigma,\varphi:X\rightarrow\mathbb R^*)&\longmapsto&\varphi^{-1}(1)\end{array}$$
$$F^-:\begin{array}{ccc}\Var_\mon^n&\longrightarrow&\AS \\ (X,\sigma,\varphi:X\rightarrow\mathbb R^*)&\longmapsto&\varphi^{-1}(-1)\end{array}$$
induce morphisms of $K_0(\AS)$-algebras (resp. $\M_\AS$-algebras)
$$F^+:K_0\rightarrow K_0(\AS)\quad\quad\quad\text{(resp. }F^+:\M\rightarrow \M_\AS\text{)}$$
$$F^-:K_0\rightarrow K_0(\AS)\quad\quad\quad\text{(resp. }F^-:\M\rightarrow \M_\AS\text{)}$$
\end{prop}

\begin{rem}
Let $A\in K_0$ then $F^{\varepsilon}\left(\frac{A}{\mathbb L^n}\right)=\frac{F^{\varepsilon}(A)}{\mathbb L_\AS^n}$ where $\varepsilon\in\{-,+\}$.
\end{rem}

\begin{prop}
Given $f:(\mathbb R^d,0)\rightarrow(\mathbb R,0)$ a Nash germ, we recover from $Z_f(T)$ the motivic zeta functions considered by G. Fichou in \cite{Fic05}\cite{Fic05-bis} applying the previous morphisms and the virtual Poincaré polynomial at each coefficient:
$$Z_f^\beta(T)=\sum_{n\ge1}\beta\left(\overline{\left[\X_n(f)\right]}\right)u^{-nd}T^n\in\mathbb Z[u,u^{-1}][[T]]$$
and
$$Z_f^{\beta,\varepsilon}(T)=\sum_{n\ge1}\beta\left(F^\varepsilon\left(\left[\X_n(f)\right]\right)\right)u^{-nd}T^n\in\mathbb Z[u,u^{-1}][[T]]$$ where $\varepsilon=+,-$. \\
\end{prop}

\subsection{Rationality of the motivic local zeta function}
\subsubsection{Change of variables key lemma}
The following lemma is a version of Denef--Loeser change of variables key lemma \cite[Lemma 3.4]{DL99-Germs} adapted for our framework. Denef--Loeser key lemma was introduced in order to generalize Kontsevich transformation rule.
\begin{lemma}[Change of variables key lemma {\cite[Lemma 4.5]{jbc1}}]\label{lem:keylemma}
Let $h:M\rightarrow\mathbb R^d$ be a proper generically one-to-one Nash map with $M$ a non-singular Nash variety. For $e\in\mathbb N$, let $$\Delta_e=\left\{\gamma\in\mathcal L_n(M),\,\ord_t\Jac_h(\gamma)=e\right\}$$ Then for $n\ge 2e$, $h_{*n}(\pi_n\Delta_e)$ is an $\AS$-set and $h_{*n}:\pi_n\Delta_e\rightarrow h_{*n}(\pi_n\Delta_e)$ is a piecewise trivial fibration\footnote{We mean that we may break $h_{*n}(\pi_n\Delta_e)$ into disjoint $\AS$ parts $B_i$ such that $h^{-1}_{*n}(B_i)$ is $\AS$ and $\AS$-isomorphic to $B_i\times\mathbb R^e$.} with fiber $\mathbb R^e$.
\end{lemma}

\subsubsection{Monomialization}\label{subsubsect:monom}
Let $f:(\mathbb R^d,0)\rightarrow(\mathbb R,0)$ be a Nash function. By Proposition \ref{prop:NashResol} there exists $h:Y\rightarrow\mathbb R^d$ a finite sequence of algebraic blowings-up such that $f\circ h$ and the jacobian determinant $\Jac_h$ have simultaneously only normal crossings.

We denote by $X_0(f)=f^{-1}(0)$ the zero locus of $f$. We denote by $(E_i)_{i\in A}$ the irreducible $\AS$ components of $h^{-1}(X_0(f))$ and we set $$N_i=\ord_{E_i}f\circ h$$ and $$\nu_i-1=\ord_{E_i}\Jac h$$

We use the usual stratification of $Y$: for $I\subset A$, we set $E_I=\cap_{i\in I}E_i$ and $\pring{E_I}=E_I\setminus\cup_{j\in A\setminus I}E_j$. Thus $Y=\sqcup_{I\subset A}\pring{E_I}$ and $E_\varnothing=Y\setminus h^{-1}(X_0)$.

For $I\neq\varnothing$, we define $U_I$ using the following classical construction \cite[3.2 Existence Theorem]{Ste99}\footnote{Using the adjunction formula, we may prove that $U_I$ is in fact the fiber product of the $U_{i|\pring{E_I}}$, $i\in I$, where $U_i$ is the complement of the null section of the normal bundle of $E_i$ in $Y$. Then $f_I:U_I\rightarrow\mathbb R^*$ is just the map induced by $f\circ h$.}.

Let $(U_k,x)$ (resp. $(U_l,x')$) be a coordinate system on $Y$ around $\pring{E_I}$ with $U_k$ (resp. $U_l$) an open $\AS$ set such that $(f\circ h)_{|U_k}=u(x)\prod_{i\in I}x_i^{N_i}$ with $u$ a unit and $E_i:x_i=0$ (resp. $(f\circ h)_{|U_l}=u(x')\prod_{i\in I}x_i'^{N_i}$ with $u'$ a unit and $E_i:x_i'=0$).

Assume that $U_k\cap U_l\neq\varnothing$. Then on $U_k\cap U_l$ we have $x'_i=\alpha^{kl}_i(x)x_i$ with $\alpha^{kl}_i$ a unit so that $f\circ\sigma(x')=u'(x')\prod_{i\in I}{x'_i}^{N_i}=\left(u'(x')\prod_{i\in I}\alpha^{kl}_i(x)^{N_i}\right)\prod_{i\in I}x_i^{N_i}$ hence $u(x)=u'(x')\prod_{i\in I}\alpha^{kl}_i(x)^{N_i}$. \\
We index a family of such $U_k$ covering $\pring{E_I}$ by $k\in K$. Then set $$T=\left\{(x,(a_i),k)\in\pring{E_I}\times(\mathbb R^*)^{|I|}\times K,\,x\in U_k\right\}$$ and $U_I=\quotient{T}{\sim}$ where $$(x,(a_i),k)\sim(y,(b_i),l)\Leftrightarrow\left\{\begin{array}{l}x=y\\b_i=\alpha_i^{kl}(x)a_i\end{array}\right.$$
 so that $p_I:U_I\rightarrow\pring{E_I}$ is a locally trivial $\AS$-fibration with fiber $(\mathbb R^*)^{|I|}$.

For $x\in U_k$ we define $f_I^k:(\pring{E_I}\cap U_k)\times(\mathbb R^*)^{|I|}\rightarrow\mathbb R^*$ by $f_I^k(x,(a_i))=u(x)\prod_{i\in I}a_i^{N_i}$. This induces a map $f_I:U_I\rightarrow\mathbb R^*$.

Let $N_I=\gcd_{i\in I}(N_i)$ then there are $\alpha_i\in\mathbb Z$ such that $N_I=\sum_{i\in I}\alpha_iN_i$. We consider the action $\tau$ of $\mathbb R^*$ on $U_I$ locally defined by $\lambda\cdot(x,(a_i)_i)=(x,(\lambda^{\alpha_i}a_i)_i)$.

\begin{defn}\label{defn:UI}
The class $[f_I:U_I\rightarrow\mathbb R^*,\tau]$ is well-defined in $K_0^{N_I}$. We shall simply denote it by $[U_I]$.
\end{defn}

\begin{prop}
$\overline{[U_I]}=(\mathbb L_{\AS}-1)^{|I|}[\pring{E_I}]\in K_0(\AS)$
\end{prop}
\begin{proof}
By Proposition \ref{prop:ptfAS}, we have $\pring{E_I}=\bigsqcup_{\alpha=1}^k X_\alpha$ with $p_I^{-1}(X_\alpha)\simeq X_\alpha\times(\mathbb R^*)^{|I|}$. Thus, in $K_0(\AS)$, we have
$$\overline{[U_I]}=\sum_{i=1}^k[p_I^{-1}(X_\alpha)]=\sum_{i=1}^k[X_\alpha](\mathbb L_{\AS}-1)^{|I|}=[\pring{E_I}](\mathbb L_{\AS}-1)^{|I|}\in K_0(\AS)$$
\end{proof}

\subsubsection{A rational expression of the motivic local zeta function}
The following theorem is similar to the rationality results for the zeta functions of \cite{DL02,Loo02,GLM,KP03,Fic05} in their respective frameworks.
\begin{thm}[Rationality formula]\label{thm:rat}
Let $f:(\mathbb R^d,0)\rightarrow(\mathbb R,0)$ be a Nash germ. Let $h:(Y,h^{-1}(0))\rightarrow(\mathbb R^d,0)$ be as in Section \ref{subsubsect:monom}. Then $$Z_f(T)=\sum_{\varnothing\neq I\subset A}\left[U_I\cap(h\circ p_I)^{-1}(0)\right]\prod_{i\in I}\frac{\mathbb L^{-\nu_i}T^{N_i}}{\mathbb 1-\mathbb L^{-\nu_i}T^{N_i}}$$
\end{thm}
\begin{proof}
The sum $$[\X_n(f)]=\sum_{e\ge1}[\X_n(f)\cap h_{*n}\Delta_e]$$ is finite since for $\gamma=h_{*n}\varphi\in\X_n(f)\cap\Delta_e$ we have $\ord_t\Jac_h(\varphi)=\sum(\nu_i-1)k_i$ and $\ord_tf\gamma=\sum N_ik_i$ where $k_i=\ord_{E_i}\varphi$. \\
By the change of variables key lemma \ref{lem:keylemma} we have $[\X_n(f)\cap h_{*n}\Delta_e]=[h_{*n}^{-1}\X_n(f)\cap\Delta_e]\mathbb L^{-e}$. \\
Next $$[h_{*n}^{-1}\X_n(f)\cap\Delta_e]\mathbb L^{-e}=\sum_{\varnothing\neq I\subset A}[h_{*n}^{-1}\X_n(f)\cap\Delta_e\cap\pring{E_I}]\mathbb L^{-e}$$ and
$$h_{*n}^{-1}\X_n(f)\cap\Delta_e\cap\pring{E_I}=\left\{\gamma\in\mathcal L_n(Y),\,\gamma(0)\in\pring{E_I}\cap h^{-1}(0),\,\sum_{i\in I}k_i(\nu_i-1)=e,\,\sum_{i\in I}k_i N_i=n,\,k_i=\ord_{E_i}\gamma\right\}$$
Let $\gamma\in\mathcal L_n(Y)$ with $\gamma(0)\in\pring{E_I}$ then $\ac_{f\circ h}(\gamma)=f_I\left(\gamma(0),(\ac\gamma_i)_{i\in I}\right)$ and $\ord_t(f\circ h)(\gamma(t))=\sum_{i\in I}k_iN_i$ where $k_i=\ord_{E_i}\gamma$. The action $\lambda\cdot\gamma(t)=\gamma(\lambda t)$ of $\mathbb R^*$ on $\mathcal L_n(Y)$ induces an action $\sigma$ of $\mathbb R^*$ on $U_I$ locally defined by $\lambda\cdot(x,(a_i)_i)=(x,(\lambda^{k_i}a_i)_i)$. We have $f_I(\lambda\cdot(x,(a_i)_i))=\lambda^nf_I(x,(a_i)_i)$ where $n=\sum_{i\in I}k_iN_i$. So the class $[f_I:U_I\rightarrow\mathbb R^*,\sigma]$ is well-defined in $K_0^n$. Let $\beta=\frac{n}{N_I}$, and consider the action $\sigma'$ of $\mathbb R^*$ on $U_I$ locally defined by $\lambda\cdot(x,(a_i)_i)=(x,(\lambda^{\beta\alpha_i}a_i)_i)$, then $[f_I,\sigma']$ is well-defined in $K_0^n$ and $[f_I,\sigma']=[f_I,\tau]$ in $K_0$. \\
We locally define $W_I$ and $g_I:W_I\rightarrow\mathbb R^*$ by $W_{I|U}=\{(x,r)\in(\pring{E_I}\cap U)\times\mathbb R^*,\,u(x)r^{N_I}\neq0\}$ and $g_{I|U}(x,r)=u(x)r^{N_I}$. We locally define an action of $\mathbb R^*$ on $W_I$ by $\lambda\cdot(x,r)=(x,\lambda^{\frac{n}{N_I}}r)$ then $[g_I:W_I\rightarrow\mathbb R^*]$ is well-defined in $K_0^n$. Next $U_I\xrightarrow[\simeq]{\psi}W_I\times\left\{a\in(\mathbb R^*)^{I},\,\prod a_i^{\frac{N_i}{N_I}}=1\right\}$ where $\psi$ and $\psi^{-1}$ are locally defined by
$$\psi(x,a)=\left(x,r=\prod a_i^{\frac{N_i}{N_I}},(r^{-\alpha_i}a_i)_i\right)$$
and
$$\psi^{-1}(x,r,a)=\left(x,(r^{\alpha_i}a_i)_i\right)$$
Thus, by \ref{item:relationaffine}, we have $[f_I,\sigma]=[f_I,\sigma']\in K_0^n$ since $\pr_{W_I}\psi(\lambda\cdot_\sigma y)=\lambda\cdot\pr_{W_I}\psi(y)=\pr_{W_I}\psi(\lambda\cdot_{\sigma'} y)$. So $[f_I,\sigma]=[f_I,\tau]$ in $K_0$ (the last one being the class $[U_I]$ of Definition \ref{defn:UI}). Then, we have

\begin{align*}
Z_f(T) &= \sum_{n\ge1}\sum_{e\ge1}\sum_{\varnothing\neq I\subset A}\sum_{\substack{k\in\mathbb N^{|I|} \\ \sum k_iN_i=n \\ \sum k_i(\nu_i-1)=e}}\left[U_I\cap(h\circ p_I)^{-1}(0)\right]\mathbb L^{-\sum k_i-e}T^n \\
       &= \sum_{\varnothing\neq I\subset A}\left[U_I\cap(h\circ p_I)^{-1}(0)\right]\sum_{k\in(\mathbb N\setminus\{0\})^{|I|}}\mathbb L^{-\sum k_i\nu_i}T^{\sum k_iN_i} \\
       &= \sum_{\varnothing\neq I\subset A}\left[U_I\cap(h\circ p_I)^{-1}(0)\right]\prod_{i\in I}\frac{\mathbb L^{-\nu_i}T^{N_i}}{\mathbb 1-\mathbb L^{-\nu_i}T^{N_i}}
\end{align*}
\end{proof}

\begin{defn}
Denote by $\M[[T]]_{\mathrm{sr}}$ the $\M$-submodule of $\M[[T]]$ spanned by $\mathbb 1$ and finite products of terms of the forms $\frac{\mathbb L^{\nu}T^N}{\mathbb1-\mathbb L^{\nu}T^N}$ and $\frac{\mathbb 1}{\mathbb 1-\mathbb L^{\nu}T^N}$ where $N\in\mathbb N_{>0}$ and $\nu\in\mathbb Z$.
\end{defn}

\begin{rem}
There exists a unique morphism of $\M$-modules $$\lim_{T\infty}:\M[[T]]_{\mathrm{sr}}\rightarrow\M$$ such that, for $(\nu_i,N_i)_{i\in I}\in(\mathbb Z\times\mathbb N_{>0})^I$ and $(\nu_j,N_j)_{j\in J}\in(\mathbb Z\times\mathbb N_{>0})^J$ where $I$ and $J$ are two finite sets, we have $$\lim_{T\infty}\left(\prod_{i\in I}\frac{\mathbb L^{\nu_i}T^{N_i}}{\mathbb1-\mathbb L^{\nu_i}T^{N_i}}\prod_{j\in J}\frac{\mathbb 1}{\mathbb 1-\mathbb L^{\nu_j}T^{N_j}}\right)=\left\{\begin{array}{ll}(-1)^{|I|}&\text{ if $J=\varnothing$}\\0&\text{ otherwise}\end{array}\right.$$
\end{rem}

\begin{defn}
Let $f:(\mathbb R^d,0)\rightarrow(\mathbb R,0)$ be a Nash germ. By \ref{thm:rat}, $$\Sf_f=-\lim_{T\infty}Z_f(T)\in\M$$ is well-defined. It is called the motivic Milnor fiber of $f$.
\end{defn}

\begin{rem}\label{rk:rat}
Given $h$ as in Section \ref{subsubsect:monom}, we have the following explicit formula $$\Sf_f=\sum_{\varnothing\neq I\subset A}(-1)^{|I|+1}\left[U_I\cap(h\circ p_I)^{-1}(0)\right]$$
\end{rem}

\begin{eg}\label{eg:x3y3}
Let $f(x,y)=x^3-y^3$, let $h$ be the blowing-up along the origin.

In the $y$-chart, $h$ is given by $h(X,Y)=(XY,Y)$ then $fh(X,Y)=Y^3(X^3-1)$ where $E_1:Y=0$ is the exceptional divisor and $E_2:X^3-1=0$ is the strict transform. We also have $\Jac_h(X,Y)=\left|\begin{array}{cc}Y&X\\0&1\end{array}\right|=Y$. And after the change of variables $\tilde Y=Y$ and $\tilde X=X-1$ we get $fh(\tilde X,\tilde Y)=\tilde Y^3\tilde X(\tilde X^2+3\tilde X+3)$ where $E_1:\tilde Y=0$ and $E_2:\tilde X=0$. \\

In the $x$-chart, $h$ is given by $h(X',Y')=(X',X'Y')$ then $fh(X',Y')=X'^3(1-Y'^3)$ where $E_1:X'=0$ and $E_2:1-Y'^3=0$. We also have $\Jac_h(X',Y')=\left|\begin{array}{cc}1&0\\Y'&X'\end{array}\right|=X'$. And after the change of variables $\tilde X'=X'$ and $\tilde Y'=Y'-1$ we get $fh(\tilde X',\tilde Y')=\tilde X'^3\tilde Y'(-\tilde Y'^2-3\tilde Y'-3)$ where $E_1:\tilde X'=0$ and $E_2:\tilde Y'=0$. \\

Thus $N_1=3$ and $N_2=1$, i.e. $f\circ h^{-1}(0)=3E_1+1E_2$. Also $\nu_1=2$ and $\nu_2=1$. \\

Notice that $(X',Y')=\left(XY,\frac1X\right)$ and thus that $(\tilde X',\tilde Y')=\left(\tilde Y(\tilde X+1),\tilde X\frac{-1}{\tilde X+1}\right)$ and $(\tilde X,\tilde Y)=\left(\tilde Y'\frac{-1}{\tilde Y'+1},\tilde X'(\tilde Y'+1)\right)$. Now we can construct $f_{\{1\}}:U_{\{1\}}\rightarrow\mathbb R^*$, $f_{\{2\}}:U_{\{2\}}\rightarrow\mathbb R^*$ and $f_{\{1,2\}}:U_{\{1,2\}}\rightarrow\mathbb R^*$ following the construction before Definition \ref{defn:UI}.

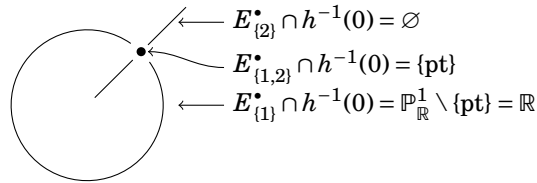
\begin{center}\begin{tikzpicture}
\draw ([shift=(55:1cm)]0,0) arc (55:395:1cm);
\node[anchor=west] at (1.8,0) {$\pring{E_{\{1\}}}\cap h^{-1}(0)=\mathbb P_\mathbb R^1\setminus\{\mathrm{pt}\}=\mathbb R$};
\path[<-] (1.2,0) edge (1.8,0);
\draw (0.1,0.1) -- (0.60,0.60);
\draw (0.80,0.80) -- (1.3,1.3);
\node[anchor=west] at (1.8,1.1) {$\pring{E_{\{2\}}}\cap h^{-1}(0)=\varnothing$};
\path[<-] (1.2,1.1) edge (1.8,1.1);
\fill (0.71,0.71) circle (.06);
\node[anchor=west] at (1.8,.5) {$\pring{E_{\{1,2\}}}\cap h^{-1}(0)=\{\mathrm{pt}\}$};
\path[<-] (0.8,0.71) edge[out=0,in=180] (1.8,.5);
\end{tikzpicture}
\end{center}

Thus
\begin{itemize}[nosep]
\item $U_{\{1\}}\cap(h\circ p_I)^{-1}(0)=(\mathbb P^1\setminus[1:1])\times\mathbb R^*$ and $f_{\{1\}}([r:s],a)=a^3(r^3-s^3)$.
\item $U_{\{2\}}\cap(h\circ p_I)^{-1}(0)=\varnothing$.
\item $U_{\{1,2\}}\cap(h\circ p_I)^{-1}(0)=(\mathbb R^*)^2$ and $f_{\{1,2\}}(a,b)=a^3b$.
\end{itemize} \ \\

Finally
\begin{align*}
Z_f(T)&=\left[f_{\{1\}}:(\mathbb P^1\setminus[1:1])\times\mathbb R^*\rightarrow\mathbb R^*\right]\frac{\mathbb L^{-2}T^3}{\mathbb1-\mathbb L^{-2}T^3}+\left[f_{\{1,2\}}:(\mathbb R^*)^2\rightarrow\mathbb R^*\right]\frac{\mathbb L^{-2}T^3}{\mathbb1-\mathbb L^{-2}T^3}\frac{\mathbb L^{-1}T}{\mathbb1-\mathbb L^{-1}T} \\
&= \mathbb L\frac{\mathbb L^{-2}T^3}{\mathbb1-\mathbb L^{-2}T^3}+(\mathbb L-\mathbb1)\frac{\mathbb L^{-2}T^3}{\mathbb1-\mathbb L^{-2}T^3}\frac{\mathbb L^{-1}T}{\mathbb1-\mathbb L^{-1}T}
\end{align*}
\qed
\end{eg}

We recover the rationality of Koike--Parusiński or Fichou zeta functions.
\begin{prop}[{\cite[Proposition 3.2\& Proposition 3.5]{Fic05}}] For $\varepsilon\in\{+,-\}$, we have
$$Z_f^\beta(T)=\sum_{\varnothing\neq I\subset A}(u-1)^{|I|}\beta\left(\pring{E_I}\cap h^{-1}(0)\right)\prod_{i\in I}\frac{u^{-\nu_i}T^{N_i}}{1-u^{-\nu_i}T^{N_i}}\in\mathbb Z[u,u^{-1}][[T]]$$
$$Z_f^{\beta,\varepsilon}(T)=\sum_{\varnothing\neq I\subset A}\beta\left(U_I\cap(h\circ p_I)^{-1}(0)\cap f_I^{-1}(\varepsilon1)\right)\prod_{i\in I}\frac{u^{-\nu_i}T^{N_i}}{1-u^{-\nu_i}T^{N_i}}\in\mathbb Z[u,u^{-1}][[T]]$$
\end{prop}
\begin{proof}
We apply the forgetful morphism (resp. $F^\pm$) to the coefficients of $$Z_f(T)=\sum_{\varnothing\neq I\subset A}\left[U_I\cap(h\circ p_I)^{-1}(0)\right]\sum_{k\in(\mathbb N\setminus\{0\})^{|I|}}\mathbb L^{-\sum k_i\nu_i}T^{\sum k_iN_i}$$
\end{proof}

\begin{eg}
Let $f(x,y)=x^3-y^3$. We deduce from Example \ref{eg:x3y3} that
$$Z_f^\beta(T)=u(u-1)\frac{u^{-2}T^3}{1-u^{-2}T^3}+(u-1)^2\frac{u^{-2}T^3}{1-u^{-2}T^3}\frac{u^{-1}T}{1-u^{-1}T}$$
$$Z_f^{\beta,+}(T)=Z_f^{\beta,-}(T)=u\frac{u^{-2}T^3}{1-u^{-2}T^3}+(u-1)\frac{u^{-2}T^3}{1-u^{-2}T^3}\frac{u^{-1}T}{1-u^{-1}T}$$
\end{eg}

\begin{prop}[{\cite[(1.1)\&(1.2)]{KP03}}]
For $\varepsilon\in\{<,>\}$, we have
$$Z_f^{\chi_c}(T)=\sum_{\varnothing\neq I\subset A}(-2)^{|I|}\chi_c\left(\pring{E_I}\cap h^{-1}(0)\right)\prod_{i\in I}\frac{(-1)^{\nu_i}T^{N_i}}{1-(-1)^{\nu_i}T^{N_i}}\in\mathbb Z[[T]]$$
$$Z_f^{\chi_c,\varepsilon}(T)=\sum_{\varnothing\neq I\subset A}\chi_c\left(U_I\cap(h\circ p_I)^{-1}(0)\cap f_I^{-1}(\mathbb R_{\varepsilon0})\right)\prod_{i\in I}\frac{(-1)^{\nu_i}T^{N_i}}{1-(-1)^{\nu_i}T^{N_i}}\in\mathbb Z[[T]]$$
\end{prop}

\begin{eg}
Let $f(x,y)=x^3-y^3$. We deduce from Example \ref{eg:x3y3} that
$$Z_f^{\chi_c}(T)=2\frac{T^3}{1-T^3}+4\frac{T^3}{1-T^3}\frac{T}{1+T}$$
$$Z_f^{\chi_c,>}(T)=Z_f^{\chi_c,<}(T)=\frac{T^3}{1-T^3}+2\frac{T^3}{1-T^3}\frac{T}{1+T}$$
\end{eg}

\section{Example: non-degenerate polynomials}
In this section we follow G. Guibert \cite[\S2.1]{Gui02} to compute the motivic local zeta function of a non-degenerate polynomial. We may find similar construction for the topological case \cite[\S5]{DL92} and for the $p$-adic case \cite{DH01}. Some ideas of these constructions go back to \cite{Kou76} and \cite{Var76}.

We may find the first adaptation in the real non-equivariant case using the virtual Poincaré polynomial in \cite{FF}.

\subsection{The Newton polyhedron of a polynomial}
We first recall some definitions related to the Newton polyhedron of a polynomial. We refer the reader to \cite[\S8]{AGV-T2} for more details.

\begin{defn}
The Newton polyhedron $\Gamma_f$ of $f(x)=\sum_{\nu\in\mathbb N^d}c_\nu x^\nu\in\mathbb R[x_1,\ldots,x_d]$ is the convex hull of $$\bigcup_{\nu\in\mathbb N^d,\,c_\nu\neq0}\nu+(\mathbb R_+)^d$$ in $(\mathbb R_+)^d$.
\end{defn}

\begin{defn}
Given a face $\tau\in\Gamma_f$, we set $f_\tau(x)=\sum_{\nu\in\tau}c_\nu x^\nu$.
\end{defn}

\begin{defn}
A polynomial $f$ is said to be non-degenerate if for every compact face\footnote{We mean the proper faces (not only the facets) and $\Gamma_f$ itself.} $\tau$ of $\Gamma_f$, the polynomials $$f_\tau,\,\frac{\partial f_\tau}{\partial x_i},\,1\le i\le d,$$ have no common zero in $(\mathbb R^*)^d$.
\end{defn}

\begin{defn}
For $k\in(\mathbb R_+)^d$, we define the supporting function $$m(k)=\inf_{x\in\Gamma_f}\{k\cdot x\}$$
\end{defn}

\begin{rem}
Actually $m(k)=\min_{x\in\Gamma_f}\{k\cdot x\}$ since we can take the infimum in the compact set $C=\Conv(\{\nu,\,c_\nu\neq0\})$ using that $\Gamma_f=C+(\mathbb R_+)^d$.
\end{rem}

\begin{defn}
For $k\in(\mathbb R_+)^d$, we define the trace of $k$ by $\tau(k)=\{x\in\Gamma_f,k\cdot x=m(k)\}$.
\end{defn}

\begin{prop}
\begin{itemize}
\item $\tau(0)=\Gamma_f$
\item For $k\neq0$, $\tau(k)$ is a proper face of $\Gamma_f$
\item $\tau(k)$ is a compact face if and only if $k\in(\mathbb R_+\setminus\{0\})^d$
\end{itemize}
\end{prop}

\begin{notation}
For $\tau$ a face of $\Gamma_f$, we define the cone of $\tau$ by $\sigma(\tau)=\left\{k\in(\mathbb R_+)^d,\tau(k)=\tau\right\}$ and we set $\tilde\sigma(\tau)=\sigma(\tau)\cap\mathbb N^d$.
\end{notation}

\begin{notation}
Let $\tau$ be a facet (i.e. a face of codimension 1), then $\tau$ is supported by a hyperplane which contains at least one point with integer coefficients. So this hyperplane has a unique equation $$\sum_{i=1}^da_ix_i=N$$ with $a_i,N\in\mathbb N$ and $\gcd\{a_i\}=1$. Thus $e^\tau=(a_1,\ldots,a_d)$ is the unique primitive vector in $\mathbb N^d\setminus\{0\}$ which is perpendicular to $\tau$.
\end{notation}

\begin{lemma}
Let $\tau$ be a proper face of $\Gamma_f$, denote by $\tau_1,\ldots,\tau_l$ the facets containing $\tau$. Then $$\sigma(\tau)=\left\{\sum_{i=1}^l\alpha_ie^{\tau_i},\,\alpha_i\in\mathbb R^*_+\right\}$$
\end{lemma}

\subsection{The motivic zeta function and Milnor fiber of a non-degenerate polynomial}
We may easily adapt the proof of \cite[Proposition 3.13]{Rai12} in order to get the following lemma.
\begin{lemma}\label{lem:nodep}
For $\tau$ a compact face of $\Gamma_f$ and $k\in\tilde\sigma(\tau)$ we define the class $\left[(\mathbb R^*)^d\setminus f_\tau^{-1}(0),\sigma_k\right]\in K_0^{m(k)}$ where the morphism is $f_\tau$ and the action is given by $\lambda\cdot_{\sigma_k}x=(\lambda^{k_i}x_i)_i$. \\
Then for $k,k'\in\sigma(\tau)$, $\left[(\mathbb R^*)^d\setminus f_\tau^{-1}(0),\sigma_k\right]=\left[(\mathbb R^*)^d\setminus f_\tau^{-1}(0),\sigma_{k'}\right]\in K_0$. We simply denote it by $\left[(\mathbb R^*)^d\setminus f_\tau^{-1}(0)\right]$.
\end{lemma}

\begin{lemma}\label{lem:homcomp}
Let $f:\mathbb R^d\rightarrow\mathbb R$ be a weighted homogeneous polynomial of weight $(k_1,\ldots,k_d;m)$ with $k_i\in\mathbb N_{>0}$ such that $f,\frac{\partial f}{\partial x_1},\ldots,\frac{\partial f}{\partial x_d}$ have no common zero in $(\mathbb R^*)^d$. For $l\ge1$, we consider $$A_l=\left\{\gamma\in\mathcal L_{m+l}(\mathbb R^d),\,\ord_t\gamma=(k_1,\ldots,k_d),\,\ord_tf\gamma=m+l\right\}$$ with the morphism $\varphi:A_l\rightarrow\mathbb R^*$ which associates to $\gamma$ the angular component of $f\gamma$ and with the action $\lambda\cdot\gamma(t)=\gamma(\lambda t)$. Thus $[A_l]\in K_0^{m+l}$ is well-defined. Then $$[A_l]=\left[f^{-1}(0)\cap(\mathbb R^*)^d\right]\mathbb L^{l(d-1)+md-\sum_{r=1}^dk_r}\in K_0^{m+l}$$
where $\left[f^{-1}(0)\cap(\mathbb R^*)^d\right]\in K_0(\AS)$.
\end{lemma}
\begin{proof}
We set $\gamma(t)=(\gamma_1(t),\ldots,\gamma_d(t))$ where $\gamma_r(t)=t^{k_r}\left(\sum_{i=0}^{m-k_r}a_{ri}t^{i}\right)$ with $a_{r0}\neq0$. \\
The coefficient of $t^m$ in $f\gamma(t)$ is $f(a_{10},\ldots,a_{d0})$ and, for $i=1,\ldots,l$, the one of $t^{m+i}$ is of the form $$g_i(a_{1i},\ldots,a_{di})-P_i$$ where $g_i$ is a linear form in $(a_{1i},\ldots,a_{di})$ which is non-zero since $\frac{\partial f}{\partial x_1},\ldots,\frac{\partial f}{\partial x_d}$ have no common zero in $(\mathbb R^*)^d$ and where $P_i$ is a polynomial in $a_{1j},\ldots,a_{dj},j<i$.

Thus $$A_l=\left\{(a_{ri})_{\substack{r=1,\ldots,d\\i=0,\ldots,m-k_r}}
\begin{array}{lll}
f(a_{10},\ldots,a_{d0})&=&0 \\
g_i(a_{1i},\ldots,a_{di})-P_i&=&0\ \ \ \text{ for i=1,\ldots,l-1}\\
g_l(a_{1l},\ldots,a_{dl})-P_l&\neq&0
\end{array}\right\}$$
Up to reordering the coordinates we may assume that the coefficient of $a_{dl}$ in $g_l$ is non-zero. Then we set
$$B_l=\left\{(\tilde a_{ri})_{\substack{r=1,\ldots,d\\i=0,\ldots,m-k_r}}
\begin{array}{lll}
f(\tilde a_{10},\ldots,\tilde a_{d0})&=&0 \\
g_i(\tilde a_{1i},\ldots,\tilde a_{di})-P_i&=&0\ \ \ \text{ for i=1,\ldots,l-1}\\
\tilde a_{dl}&\neq&0
\end{array}\right\}$$
and we define the following $\AS$-bijection over $\mathbb R^*$
$$\xymatrix{A_l \ar[rr] \ar[rd]_{x^{m+l}\circ(g_l-P_l)}&&B_l \ar[ld]^{x^{m+l}\circ\pr_{\tilde a_{dl}}} \\ &\mathbb R^*&}$$
by $\tilde a_{ri}=a_{ri}$ if $(r,i)\neq(d,l)$ and $\tilde a_{dl}=g_l(a_{1l},\ldots,a_{dl})-P_l$

Therefore $$[A_l]=[B_l]=\left[f^{-1}(0)\cap(\mathbb R^*)^d\right]\left(\mathbb L^{d-1}\right)^{l-1}\mathbb L^{d-1}\mathbb1\mathbb L^{\sum_{r=1}^d(m-k_r)}=\left[f^{-1}(0)\cap(\mathbb R^*)^d\right]\mathbb L^{l(d-1)-md-\sum_{r=1}^dk_r}$$
\end{proof}

\begin{defn}
For $m\in\mathbb Z$ we define $\mathcal F^m\M$ as the subgroup of $\M$ spanned by elements $[S]\mathbb L^{-i}$ with $i-\dim S\ge m$. We denote by $\widehat\M$ the completion of $\M$ with respect to the filtration $\mathcal F^\cdot\M$.
\end{defn}

\begin{rem}
The ring $\widehat\M$ allows us to handle terms of the form $\sum_{i}\mathbb L^{-ki}$ that may appear in the following formula.
\end{rem}

\begin{thm}\label{thm:Guibert}
Let $f$ be non-degenerate, then
$$Z_f(T)=\sum_{\tau\text{ compact face}}\left(\left[(\mathbb R^*)^d\setminus f_\tau^{-1}(0)\right]+\left[f_\tau^{-1}(0)\cap(\mathbb R^*)^d\right]\frac{\mathbb L^{-1}T}{\mathbb 1-\mathbb L^{-1}T}\right)\sum_{k\in\tilde\sigma(\tau)}\mathbb L^{-\sum k_i}T^{m(k)}\in\widehat\M[[T]]$$
where $\left[(\mathbb R^*)^d\setminus f_\tau^{-1}(0)\right]$ is defined in Lemma \ref{lem:nodep} and $\left[f_\tau^{-1}(0)\cap(\mathbb R^*)^d\right]\in K_0(\AS)$.
\end{thm}
\begin{proof}
We first notice that $$(\mathbb N\setminus\{0\})^d=\bigsqcup_{\tau\text{ compact face of }\Gamma_f}\tilde\sigma(\tau)$$ Thus\footnote{The condition $\gamma(0)=0$ is satisfied since $k\in(\mathbb N\setminus\{0\})^d$.}
\begin{align*}
Z_f(T) &= \sum_{n\ge1}\left[\X_n(f)\right]\mathbb L^{-nd}T^n \\
&= \sum_{\tau\text{ compact face}}\sum_{k\in\tilde\sigma(\tau)}\sum_{n\ge m(k)}\left[\gamma\in\mathcal L_n(\mathbb R^d),\,\ord_t\gamma=k,\,\ord_tf\gamma=n\right]\mathbb L^{-nd}T^n \\
&= \sum_{\tau}\vast(\sum_{k\in\tilde\sigma(\tau)}\left[\gamma\in\mathcal L_{m(k)}(\mathbb R^d),\,\ord_t\gamma=k,\,\ord_tf\gamma=m(k)\right]\mathbb L^{-m(k)d}T^{m(k)} \\
&\quad\quad\quad+\sum_{k\in\tilde\sigma(\tau)}\sum_{l\ge1}\left[\gamma\in\mathcal L_{m(k)+l}(\mathbb R^d),\,\ord_t\gamma=k,\,\ord_tf\gamma=m(k)+l\right]\mathbb L^{-(m(k)+l)d}T^{m(k)+l}\vast) \\
&= \sum_\tau\left(Z_\tau^=(T)+Z_\tau^>(T)\right)
\end{align*}
Fix $\tau$ a compact face of $\Gamma_f$ and $k\in(\mathbb N\setminus\{0\})^d$ such that $\tau(k)=\tau$. Let $\gamma\in\mathcal L_{m(k)}(\mathbb R^d)$ satisfying $\ord_t\gamma=k$ and $\ord_tf\gamma=m(k)$. Then $\gamma(t)=(t^{k_i}a_i(t))_i$ with $a_i(0)\neq0$ and $$f\gamma(t)=\sum_\nu c_\nu a(t)^\nu t^{k\cdot\nu}=f_\tau(a(0))t^{m(k)}+t^{m(k)+1}R(t)$$
Thus $\ord_tf\gamma=\ord_tf_\tau\gamma$ and $\ac f\gamma=\ac f_\tau\gamma$ and 
$$Z_\tau^=(T)=\sum_{k\in\tilde\sigma(\tau)}\left[(\mathbb R^*)^d\setminus f_\tau^{-1}(0),\sigma_k\right]\mathbb L^{-\sum k_i}T^{m(k)}$$
where the morphism is $f_\tau$ and the action $\sigma_k$ is the one induced by the action on $\mathcal L_{m(k)}(\mathbb R^d)$, i.e. $\lambda\cdot_{\sigma_k}(x_1,\ldots,x_d)=(\lambda^{k_1}x_1,\ldots,\lambda^{k_d}x_d)$. Thus $\left[(\mathbb R^*)^d\setminus f_\tau^{-1}(0),\sigma_k\right]$ is well-defined in $K_0^{m(k)}$. By the lemma \ref{lem:nodep}, we get
$$Z_\tau^=(T)=\left[(\mathbb R^*)^d\setminus f_\tau^{-1}(0)\right]\sum_{k\in\tilde\sigma(\tau)}\mathbb L^{-\sum k_i}T^{m(k)}$$

Now let $l\ge 1$ and $\gamma\in\mathcal L_{m(k)+l}(\mathbb R^d)$ satisfying $\ord_t\gamma=k$ and $\ord_tf\gamma=m(k)+l$. We set $\gamma(t)=(t^{k_1}a_1(t),\ldots,t^{k_d}a_d(t))$ with $a_i(0)\neq0$. By Lemma \ref{lem:homcomp} we get
\begin{align*}
Z_\tau^>(T)&=\sum_{k\in\tilde\sigma(\tau)}\sum_{l\ge1}\left[f_\tau^{-1}(0)\cap(\mathbb R^*)^d\right]\mathbb L^{-\sum k_i-l}T^{m(k)+l}\\
&=\left[f_\tau^{-1}(0)\cap(\mathbb R^*)^d\right]\frac{\mathbb L^{-1}T}{\mathbb 1-\mathbb L^{-1}T}\sum_{k\in\tilde\sigma(\tau)}\mathbb L^{-\sum k_i}T^{m(k)}\\
&=\left[f_\tau^{-1}(0)\cap(\mathbb R^*)^d\right]\frac{\mathbb L^{-1}T}{\mathbb 1-\mathbb L^{-1}T}\sum_{k\in\tilde\sigma(\tau)}\mathbb L^{-\sum k_i}T^{m(k)}
\end{align*}
\end{proof}

\begin{cor}
If $f$ is non-degenerate then
$$S_f=-\sum_{\tau\text{ compact face}}(-1)^{d-\dim\tau}\left(\left[(\mathbb R^*)^d\setminus f_\tau^{-1}(0)\right]-\left[f_\tau^{-1}(0)\cap(\mathbb R^*)^d\right]\cdot\mathbb 1\right)\in\widehat\M$$
\end{cor}
\begin{proof}
It's a direct application of \cite[pp1006--1007]{Den87} \cite[Lemme 2.1.5]{Gui02}, noticing that $\dim\tau+\dim\sigma(\tau)=d$. 
\end{proof}

\begin{eg}
Let $f(x,y)=x^3-y^3$. The Newton polyhedron of $f$ has 3 compact faces:
\begin{enumerate}[nosep]
\item $\tau_1=\left\{\lambda(0,3)+(1-\lambda)(3,0),\,\lambda\in[0,1]\right\}$ with $\tilde\sigma(\tau_1)=\mathbb N_{>0}(1,1)$.
\item $\tau_2=\left\{(0,3)\right\}$ with $\tilde\sigma(\tau_2)=(2,1)+\mathbb N(1,1)+\mathbb N(0,1)$.
\item $\tau_3=\left\{(3,0)\right\}$ with $\tilde\sigma(\tau_3)=(1,2)+\mathbb N(1,1)+\mathbb N(0,1)$.
\end{enumerate}
Thus
\begin{align*}
Z_f(T)&=\left([(x,y)\in(\mathbb R^*)^2,x^3-y^3\neq0]+\overline{[(x,y)\in(\mathbb R^*)^2,x^3-y^3=0]}\frac{\mathbb L^{-1}T}{\mathbb1-\mathbb L^{-1}T}\right)\frac{\mathbb L^{-2}T^3}{\mathbb1-\mathbb L^{-2}T^3} \\
&\quad\quad\quad+\left[(x,y)\in(\mathbb R^*)^2\mapsto x^3\right]\frac{\mathbb1}{\mathbb L-\mathbb1}\frac{\mathbb L^{-2}T^3}{\mathbb1-\mathbb L^{-2}T^3}+\left[(x,y)\in(\mathbb R^*)^2\mapsto -y^3\right]\frac{\mathbb1}{\mathbb L-\mathbb1}\frac{\mathbb L^{-2}T^3}{\mathbb1-\mathbb L^{-2}T^3} \\
&=\left([(x,y)\in(\mathbb R^*)^2,x^3-y^3\neq0]+\left[x\in\mathbb R^*\mapsto x^3\right]+\left[y\in\mathbb R^*\mapsto -y^3\right]+(\mathbb L-\mathbb1)\frac{\mathbb L^{-1}T}{\mathbb1-\mathbb L^{-1}T}\right)\frac{\mathbb L^{-2}T^3}{\mathbb1-\mathbb L^{-2}T^3} \\
&=\left((\mathbb L-\mathbb 1-\mathbb1)+\mathbb1+\mathbb1+(\mathbb L-\mathbb1)\frac{\mathbb L^{-1}T}{\mathbb1-\mathbb L^{-1}T}\right)\frac{\mathbb L^{-2}T^3}{\mathbb1-\mathbb L^{-2}T^3} \\
&=\left(\mathbb L+(\mathbb L-\mathbb1)\frac{\mathbb L^{-1}T}{\mathbb1-\mathbb L^{-1}T}\right)\frac{\mathbb L^{-2}T^3}{\mathbb1-\mathbb L^{-2}T^3}
\end{align*}
The third equality comes from the fact that the following diagram commutes
$$\xymatrix{
\left\{(x,y)\in(\mathbb R^*)^2,x^3-y^3\neq0\right\} \ar[rr]^\psi \ar[rd]_{f} & & \mathbb R^*\times\mathbb R\setminus\{0,1\} \ar[ld]^{\tilde f} \\
& \mathbb R^* &
}$$
where $f(x,y)=x^3-y^3$, $\tilde f(a,b)=a$, $\psi(x,y)=\left(x^3-y^3,\frac{y}{x}\right)$ and $\psi^{-1}(a,b)=\left(\left(\frac{a}{1-b^3}\right)^{\frac{1}{3}},\left(\frac{a}{1-b^3}\right)^{\frac{1}{3}}b\right)$.
\end{eg}

\section{A convolution formula for the motivic local zeta function}\label{sect:TS}
The goal of this section is to express $Z_{f_1\oplus f_2}(T)$ in terms of $Z_{f_1}(T)$ and $Z_{f_2}(T)$ where the $f_i:(\mathbb R^{d_i},0)\rightarrow(\mathbb R,0)$ are two Nash germs and $f_1\oplus f_2(x_1,x_2)=f_1(x_1)+f_2(x_2)$.

The idea of the proof is the following, given two arcs $\gamma_i\in\mathcal L(\mathbb R^d,0)$, we have two cases to distinguish: either $\ord_tf_1\gamma_1\neq\ord_tf_2\gamma_2$, let say $\ord_tf_1\gamma_1<\ord_tf_2\gamma_2$, and then $\ord_t(f_1\oplus f_2)(\gamma_1,\gamma_2)=\ord_tf_1\gamma_1$ and $\ac(f_1\oplus f_2)(\gamma_1,\gamma_2)=\ac f_1\gamma_1$ or $\ord_tf_1\gamma_1=\ord_tf_2\gamma_2$ and then two phenomena may appear. In this case, either $\ac f_1\gamma_1+\ac f_2\gamma_2\neq0$ and then $\ord_t(f_1\oplus f_2)(\gamma_1,\gamma_2)=\ord_tf_1\gamma_1=\ord_tf_2\gamma_2$ and $\ac(f_1\oplus f_2)(\gamma_1,\gamma_2)=\ac f_1\gamma_1+\ac f_2\gamma_2$ or $\ac f_1\gamma_1+\ac f_2\gamma_2=0$ and then $\ord_t(f_1\oplus f_2)(\gamma_1,\gamma_2)>\ord_tf_1\gamma_1=\ord_tf_2\gamma_2$.

In \cite{DL99-TS} or \cite{Loo02}, the authors work with an equivariant Grothendieck ring over $\mathbb C$ with actions of the roots of unity (with the additional hypothesis that $\ac f\gamma=1$ in $\X_n(f)$). Then they consider the motives of $$\quotient{\{x^n+y^n=\varepsilon\}\times(\X_n(f_1)\times\X_n(f_2))}{(\lambda\cdot(x,y),(\gamma_1,\gamma_2))\sim((x,y),\lambda\cdot(\gamma_1,\gamma_2))}$$ where $\lambda\in\mu_n$ and $\varepsilon\in\{0,1\}$, to handle the case $\ord_tf_1\gamma_1=\ord_tf_2\gamma_2$.

The lack of real roots of unity and the quotient don't allow us to adapt these constructions. However our ring $K_0$, adapted from the one of \cite{GLM}, remembers the angular component morphisms $\ac_f:\gamma\mapsto\ac f\gamma$. Thus, following \cite[\S5.1]{GLM}, we may define a convolution product in order to get a convolution formula\footnote{We may also notice that the convolution \cite[\S5.1]{GLM} is compatible with the one of \cite{Loo02} by \cite[(5.1.8)]{GLM}.}. The definition of this convolution product is motivated by the previous discussion.

\begin{notation}
We denote by $*:K_0^m\times K_0^n\rightarrow K_0^{nm}$ the unique $K_0(\AS)$-bilinear map satisfying
$$[X,\sigma,\varphi]*[Y,\tau,\psi]=-[Z_1,\mu_1,f_1]+[Z_2,\mu_2,f_2]$$
where
$$\left\{\begin{array}{l}Z_1=X\times Y\setminus(\varphi+\psi)^{-1}(0) \\ f_1=\varphi+\psi \\ \lambda\cdot_{\mu_1}(x,y)=(\lambda^n\cdot_\sigma x,\lambda^m\cdot_\tau y)\end{array}\right.$$
and
$$\left\{\begin{array}{l}Z_2=(\varphi+\psi)^{-1}(0)\times\mathbb R^* \\ f_2=\pr_{\mathbb R^*} \\ \lambda\cdot_{\mu_2}(x,y,r)=(\lambda^n\cdot_\sigma x,\lambda^m\cdot_\tau y,\lambda^{mn}r)\end{array}\right.$$
\end{notation}

\begin{rem}
The map $*:K_0^m\times K_0^n\rightarrow K_0^{nm}$ induces a $K_0(\AS)$-bilinear map $*:K_0\times K_0\rightarrow K_0$ and a $\M_\AS$-bilinear map $*:\M\times\M\rightarrow\M$.
\end{rem}

\begin{prop}
The convolution product $*$ in $K_0$ (resp. $\M$) is commutative and associative. The class $\mathbb1$ is the unit of this product.
\end{prop}
\begin{proof}
We easily adapt the proof of \cite[Proposition 5.2]{GLM}.
\end{proof}

\begin{lemma}\label{lem:massless}
Let $g:(\mathbb R^{d},0)\rightarrow(\mathbb R,0)$ be a Nash germ. Then, in $K_0(\AS)$, we have the relation $$\left[\left\{\gamma\in\mathcal L_n(\mathbb R^d,0),\,\ord_tg\gamma>n\right\}\right]=\mathbb L_{\AS}^{nd}\left(1-\displaystyle\sum_{i=1}^n[\X_i(g)]\mathbb L_{\AS}^{-id}\right)$$
\end{lemma}
\begin{proof}
$\mathcal L_n(\mathbb R^d,0)=\bigsqcup_{i=1}^n(\pi^n_i)^{-1}(\X_i(g))\sqcup\left\{\gamma\in\mathcal L_n(\mathbb R^d,0),\,\ord_tg\gamma>n\right\}$. \\
So $\mathbb L_{\AS}^{nd}=\left(\displaystyle\sum_{i=1}^n[\X_i(g)]\mathbb L_{\AS}^{(n-i)d}\right)+\left[\left\{\gamma\in\mathcal L_n(\mathbb R^d,0),\,\ord_tg\gamma>n\right\}\right]$. \\
Finally, $\left[\left\{\gamma\in\mathcal L_n(\mathbb R^d,0),\,\ord_tg\gamma>n\right\}\right]=\mathbb L_{\AS}^{nd}-\displaystyle\sum_{i=1}^n[\X_i(g)]\mathbb L_{\AS}^{(n-i)d}$
\end{proof}

\begin{defn}
We define the motivic naive local zeta function by $$Z_f^{\mathrm{naive}}(T)=\sum_{n\ge1}\overline{\left[\X_n(f)\right]}\mathbb L^{-nd}T^n\in\M[[T]]$$
\end{defn}

\begin{defn}
Let $f:(\mathbb R^d,0)\rightarrow(\mathbb R,0)$ be a Nash germ. We define the modified zeta function by 
$$\tilde Z_f(T)=\sum_{n\ge1}\left[\Y_n(f)\right]\mathbb L^{-nd}T^n\in\M[[T]]$$
where
$$\left[\Y_n(f)\right]=\left[\X_n(f)\right]-\overline{\left[\gamma\in\mathcal L_n(\mathbb R^d,0),\,\ord_tf\gamma>n\right]}\cdot\mathbb1$$
\end{defn}

\begin{prop}
$\displaystyle\tilde Z_f(T)=Z_f(T)-\frac{\mathbb1-Z_f^{\mathrm{naive}}(T)}{\mathbb1-T}+\mathbb 1$
\end{prop}
\begin{proof}
Lemma \ref{lem:massless} allows to rewrite $\left[\Y_n(f)\right]$ as follows
$$\left[\Y_n(f)\right] = \left[\X_n(f)\right]-\mathbb L^{nd}\left(\mathbb1-\sum_{i=1}^n\overline{\left[\X_i(f)\right]}\mathbb L^{-id}\right)$$
Then
\begin{align*}
\tilde Z_f(T)&=\sum_{n\ge1}\left[\Y_n(f)\right]\mathbb L^{-nd}T^n\\
&=\sum_{n\ge1}\left[\X_n(f)\right]\mathbb L^{-nd}T^n-\sum_{n\ge1}T^n+\sum_{n\ge1}\sum_{i=1}^n\overline{\left[\X_i(f)\right]}\mathbb L^{-id}T^n\\
&=Z_f(T)-\frac{T}{\mathbb 1-T}+\sum_{i\ge1}\sum_{n\ge i}\overline{\left[\X_i(f)\right]}\mathbb L^{-id}T^n \\
&= Z_f(T)-\frac{T-\mathbb1+\mathbb1}{\mathbb1-T}+\frac{\mathbb1}{\mathbb1-T}\sum_{i\ge1}\overline{\left[\X_i(f)\right]}\mathbb L^{-id}T^i \\
&= Z_f(T)-\frac{\mathbb1-Z_f^{\mathrm{naive}}(T)}{\mathbb1-T}+\mathbb 1
\end{align*}
\end{proof}

\begin{cor}\label{cor:limzetatilde}
$\displaystyle-\lim_{T\infty}\tilde Z_f(T)=\Sf_f-\mathbb 1$
\end{cor}

\begin{rem}
Applying the forgetful morphism or the morphisms $F^>,F^<$ and the Euler characteristic with compact support to the coefficients of $\tilde Z_f(T)$ we recover the modified zeta functions of S. Koike and A. Parusiński \cite{KP03}.
$$\tilde Z_f^{\chi_c,>}(T)=\sum_{n\ge1}\chi_c\left(\gamma\in\mathcal L_n(\mathbb R^d,0),\,f\gamma(t)=ct^n\mod t^{n+1},\,c\ge0\right)(-1)^{nd}T^n=\frac{1-Z_f^{\chi_c}(T)}{1-T}-1+Z_f^{\chi_c,>}(T)\in\mathbb Z[[T]]$$
$$\tilde Z_f^{\chi_c,<}(T)=\sum_{n\ge1}\chi_c\left(\gamma\in\mathcal L_n(\mathbb R^d,0),\,f\gamma(t)=ct^n\mod t^{n+1},\,c\le0\right)(-1)^{nd}T^n=\frac{1-Z_f^{\chi_c}(T)}{1-T}-1+Z_f^{\chi_c,<}(T)\in\mathbb Z[[T]]$$
\begin{align*}
\hspace{-1.6cm}\tilde Z_f^{\chi_c}(T)&=\sum_{n\ge1}\left(\chi_c\left(\gamma\in\mathcal L_n(\mathbb R^d,0),\,\ord_tf\gamma(t)=n\right)+2\chi_c\left(\gamma\in\mathcal L_n(\mathbb R^d,0),\,\ord_tf\gamma>n\right)\right)(-1)^{nd}T^n\\
&=\sum_{n\ge1}\left(\chi_c\left(\gamma\in\mathcal L_n(\mathbb R^d,0),\,\ord_tf\gamma(t)\ge n\right)+\chi_c\left(\gamma\in\mathcal L_n(\mathbb R^d,0),\,\ord_tf\gamma>n\right)\right)(-1)^{nd}T^n\\
&=\sum_{n\ge1}\Big(\chi_c\left(\gamma\in\mathcal L_n(\mathbb R^d,0),\,f\gamma(t)=ct^n\mod t^{n+1},\,c\ge0\right)\\
&\quad\quad\quad+\chi_c\left(\gamma\in\mathcal L_n(\mathbb R^d,0),\,f\gamma(t)=ct^n\mod t^{n+1},\,c\le0\right)\Big)(-1)^{nd}T^n\\
&=\tilde Z_f^{\chi_c,>}(T)+\tilde Z_f^{\chi_c,<}(T)\\
&=2\frac{1-Z_f^{\chi_c}(T)}{1-T}-2+Z_f^{\chi_c}(T)\in\mathbb Z[[T]]
\end{align*}
\end{rem}

\begin{eg}\label{eg:monomial}
Let $f_k^\varepsilon(x)=\varepsilon x^k$ where $\varepsilon\in\{\pm1\}$, then
\begin{align*}
\hspace{-1cm}\tilde Z_{f_k^\varepsilon}(T)&=-\sum_{r=1}^{k-1}T^{r}+\sum_{q\ge1}\left(([f_k^\varepsilon:\mathbb R^*\rightarrow\mathbb R^*]-\mathbb1)\mathbb L^{-q}T^{kq}-\sum_{r=1}^{k-1}\mathbb L^{-q}T^{kq+r}\right)\\
&= -\sum_{q\ge0}\sum_{r=1}^{k-1}\mathbb L^{-q}T^{kq+r}-\sum_{q\ge1}(\mathbb1-[f_k^\varepsilon:\mathbb R^*\rightarrow\mathbb R^*])\mathbb L^{-q}T^{kq} \\
&= -T-\cdots-T^{k-1}-\left(\mathbb1-[f_k^\varepsilon]\right)\mathbb L^{-1}T^k-\mathbb L^{-1}T^{k+1}-\cdots-\mathbb L^{-1}T^{2k-1}-\left(\mathbb1-[f_k^\varepsilon]\right)\mathbb L^{-2}T^{2k}-\mathbb L^{-2}T^{2k+1}-\cdots
\end{align*}
Indeed $$\left[\X_{kq+r}(f_k^\varepsilon)\right]\mathbb L^{-(kq+r)}=\left\{\begin{array}{ll}[f_k^\varepsilon:\mathbb R^*\rightarrow\mathbb R^*]\mathbb L^{-q}&\text{ if $r=0$}\\0&\text{ otherwise}\end{array}\right.$$
and $$\overline{\left[\X_{kq+r}(f_k^\varepsilon)\right]}\mathbb L^{-(kq+r)}=\left\{\begin{array}{ll}(\mathbb L-\mathbb 1)\mathbb L^{-q}&\text{ if $r=0$}\\0&\text{ otherwise}\end{array}\right.$$
Thus
\begin{align*}
\left[\Y_{kq+r}(f_k^\varepsilon)\right]\mathbb L^{-(kq+r)}&=\left\{\begin{array}{ll}
\displaystyle[f_k^\varepsilon:\mathbb R^*\rightarrow\mathbb R^*]\mathbb L^{-q}+\sum_{i=1}^q(\mathbb L-\mathbb 1)\mathbb L^{-i}-\mathbb1&\text{ if $r=0$}\\
\displaystyle\sum_{i=1}^q(\mathbb L-\mathbb 1)\mathbb L^{-i}-\mathbb1&\text{ otherwise}
\end{array}\right.\\
&=\left\{\begin{array}{ll}
\displaystyle([f_k^\varepsilon:\mathbb R^*\rightarrow\mathbb R^*]-\mathbb1)\mathbb L^{-q}&\text{ if $r=0$}\\
\displaystyle-\mathbb L^{-q}&\text{ otherwise}
\end{array}\right.
\end{align*}
\end{eg}

\begin{defn}
We define the motivic naive modified zeta function by $$\tilde Z_f^{\mathrm{naive}}(T)=\sum_{n\ge1}\overline{\left[\Y_n(f)\right]}\mathbb L^{-nd}T^n\in\M[[T]]$$
\end{defn}

\begin{prop}
$\displaystyle\frac{\mathbb L-\tilde Z_f^{\mathrm{naive}}(T)}{\mathbb L-T}=\frac{\mathbb 1-Z_f^{\mathrm{naive}}(T)}{\mathbb 1-T}$
\end{prop}
\begin{proof}
From $$\tilde Z_f(T)=Z_f(T)-\frac{\mathbb1-Z_f^{\mathrm{naive}}(T)}{\mathbb1-T}+\mathbb 1$$ we deduce
$$\tilde Z_f^{\mathrm{naive}}(T)=Z_f^{\mathrm{naive}}(T)-(\mathbb L-\mathbb 1)\frac{\mathbb1-Z_f^{\mathrm{naive}}(T)}{\mathbb1-T}+\mathbb L-\mathbb1$$
$$\tilde Z_f^{\mathrm{naive}}(T)-\mathbb L=(Z_f^{\mathrm{naive}}(T)-\mathbb 1)+(Z_f^{\mathrm{naive}}(T)-\mathbb 1)\frac{\mathbb L-\mathbb 1}{\mathbb1-T}$$
$$\tilde Z_f^{\mathrm{naive}}(T)-\mathbb L=(Z_f^{\mathrm{naive}}(T)-\mathbb 1)\frac{\mathbb L-T}{\mathbb 1-T}$$
\end{proof}

\begin{rem}
We recover the following formula of \cite[p2070]{KP03} $$\frac{1-Z_f^{\chi_c}(T)}{1-T}=\frac{1+\tilde Z_f^{\chi_c}(T)}{1+T}$$
\end{rem}

\begin{cor}
We may compute $\tilde Z_f(T)$ from $Z_f(T)$ and $Z_f(T)$ from $\tilde Z_f(T)$. \\
More precisely, we have 
$$\tilde Z_f(T)=Z_f(T)-\frac{\mathbb1-Z_f^{\mathrm{naive}}(T)}{\mathbb1-T}+\mathbb 1$$
and
$$Z_f(T)=\tilde Z_f(T)+\frac{\mathbb L-\tilde Z_f^{\mathrm{naive}}(T)}{\mathbb L-T}-\mathbb1$$
\end{cor}

\begin{thm}\label{thm:TSzeta}
Let $f_1:(\mathbb R^{d_1},0)\rightarrow(\mathbb R,0)$ and $f_2:(\mathbb R^{d_2},0)\rightarrow(\mathbb R,0)$ be two Nash germs. Then $$\tilde Z_{f_1\oplus f_2}(T)=-\tilde Z_{f_1}(T)\circledast\tilde Z_{f_2}(T)$$ where the product $\circledast$ is the Hadamard product which consists in applying the convolution product coefficientwise.
\end{thm}

\begin{rem}
Similar formulas are known when the angular component is fixed to be $1$ with an action of the roots of unity \cite[Main Theorem 4.2.4]{DL99-TS} or for the Euler characteristic with compact support \cite[Theorem 2.3]{KP03}.
\end{rem}

\begin{rem}
In the definition of the modified zeta function, it could have been multiplied by a factor $(-1)^d$ in order to avoid the sign in Theorem \ref{thm:TSzeta}.
\end{rem}

\begin{eg}
Let $f(x,y)=x^3-y^3$. We deduce from Example \ref{eg:monomial} and Theorem \ref{thm:TSzeta} that
$$\tilde Z_f(T) = -\sum_{q\ge0}\sum_{r=1}^2\mathbb L^{-2q}T^{3q+r} = -\frac{T+T^2}{\mathbb1-\mathbb L^{-2}T^3}$$
We recover that
\begin{align*}
Z_f(T) &\displaystyle= \tilde Z_f(T)+\frac{\mathbb L-\tilde Z_f^{\mathrm{naive}}(T)}{\mathbb L-T}-\mathbb 1 \\
&\displaystyle=-\frac{T+T^2}{\mathbb1-\mathbb L^{-2}T^3}+\frac{\mathbb L+(\mathbb L-\mathbb 1)\frac{T+T^2}{\mathbb1-\mathbb L^{-2}T^3}}{\mathbb L-T}-\mathbb 1 \\
&\displaystyle=\frac{T^3-\mathbb L^{-2}T^4}{(\mathbb L-T)(\mathbb 1-\mathbb L^{-2}T^3)} \\
&\displaystyle=\frac{\mathbb L^{-1}T^3-\mathbb L^{-3}T^4}{(\mathbb1-\mathbb L^{-1}T)(\mathbb 1-\mathbb L^{-2}T^3)} \\
&\displaystyle=\mathbb L\frac{\mathbb L^{-2}T^3}{\mathbb1-\mathbb L^{-2}T^3}+(\mathbb L-\mathbb 1)\frac{\mathbb L^{-1}T}{\mathbb1-\mathbb L^{-1}T}\frac{\mathbb L^{-2}T^3}{\mathbb1-\mathbb L^{-2}T^3}
\end{align*}
\qed
\end{eg}

\begin{lemma}[{\cite[Lemma 7.6]{Loo02}\cite[Proposition 5.1.2]{DL99-TS}}]\label{lem:limconv}
Let $Z_1(T),Z_2(T)\in\M[[T]]_{\mathrm{sr}}$ then $$\lim_{T\infty}Z_1(T)\circledast Z_2(T)=-\left(\lim_{T\infty}Z_1(T)\right)*\left(\lim_{T\infty}Z_2(T)\right)$$
\end{lemma}

\begin{cor}[Motivic Thom--Sebastiani formula]
$\displaystyle\Sf_{f_1\oplus f_2}=-\Sf_{f_1}*\Sf_{f_2}+\Sf_{f_1}+\Sf_{f_2}$
\end{cor}
\begin{proof}
$$\begin{array}{rcll}
\Sf_{f_1\oplus f_2}-\mathbb 1&=&-\lim_{T\infty}\tilde Z_{f_1\oplus f_2}(T)&\text{ by Corollary \ref{cor:limzetatilde}}\\
&=&\lim_{T\infty}\left(\tilde Z_{f_1}(T)\circledast\tilde Z_{f_2}(T)\right)&\text{ by Theorem \ref{thm:TSzeta}} \\
&=&-\left(\lim_{T\infty}\tilde Z_{f_1}(T)\right)*\left(\lim_{T\infty}\tilde Z_{f_2}(T)\right)&\text{ by Lemma \ref{lem:limconv}} \\
&=&-\left(-\Sf_{f_1}+\mathbb 1\right)*\left(-\Sf_{f_2}+\mathbb 1\right)&\text{ by Corollary \ref{cor:limzetatilde}}\\
&=&-\Sf_{f_1}*\Sf_{f_2}+\Sf_{f_1}+\Sf_{f_2}-\mathbb 1&
\end{array}$$
\end{proof}

\begin{proof}[Proof of Theorem \ref{thm:TSzeta}]
In order to shorten the formulas, we set $\mathcal L(M,A)=\left\{\gamma\in\mathcal L(M),\,\gamma(0)\in A\right\}$. \\

Our goal is to compute, for $n\in\mathbb N\setminus\{0\}$, $\X_n(f_1\oplus f_2)$, so let $(\gamma_1,\gamma_2)\in\X_n(f_1\oplus f_2)$, i.e. $\gamma_i\in\mathcal L_n(\mathbb R^{d_i},0)$ such that $\ord_t(f_1\gamma_1+f_2\gamma_2)=n$. \\

We first restrict to the case $\ord_tf_1\gamma_1\neq\ord_tf_2\gamma_2$. Then either $$n=\ord_tf_1\gamma_1<\ord_tf_2\gamma_2\text{ with }\ac\left((f_1\oplus f_2)(\gamma_1,\gamma_2)\right)=\ac(f_1\gamma_1)$$ or $$\ord_tf_1\gamma_1>\ord_tf_2\gamma_2=n\text{ with }\ac\left((f_1\oplus f_2)(\gamma_1,\gamma_2)\right)=\ac(f_2\gamma_2)$$ \\
Therefore, if we set $$\X_n^{\neq}(f_1\oplus f_2)=\left\{(\gamma_1,\gamma_2)\in\mathcal L_n(\mathbb R^{d_1+d_2},0),\,\ord_tf_1\gamma_1\neq\ord_tf_2\gamma_2,\,\ord_t(f_1\gamma_1+f_2\gamma_2)=n\right\}$$ with the natural action on jets and the natural morphism induced by the angular component, we get
\begin{align*}
\left[\X_n^{\neq}(f_1\oplus f_2)\right]&=\overline{\left[\gamma_2\in\mathcal L_n(\mathbb R^{d_2},0),\,\ord_tf_2\gamma_2>n\right]}\left[\X_n(f_1)\right]+\overline{\left[\gamma_1\in\mathcal L_n(\mathbb R^{d_1},0),\,\ord_tf_1\gamma_1>n\right]}\left[\X_n(f_2)\right]\\
&=\left(1-\sum_{\omega=1}^n\overline{[\X_\omega(f_2)]}\mathbb L^{-\omega d_2}\right)[\X_n(f_1)]\mathbb L^{nd_2}+\left(1-\sum_{\omega=1}^n\overline{[\X_\omega(f_1)]}\mathbb L^{-\omega d_1}\right)[\X_n(f_2)]\mathbb L^{nd_1}
\end{align*}
where the overline means that the class is in $K_0(\AS)$ and that we use the scalar multiplication. The second equality comes from Lemma \ref{lem:massless}. \\

It remains to manage the case $\ord_tf_1\gamma_1=\ord_tf_2\gamma_2$. Set $$\X_n^=(f_1\oplus f_2)=\left\{(\gamma_1,\gamma_2)\in\mathcal L_n(\mathbb R^{d_1+d_2},0),\,\ord_tf_1\gamma_1=\ord_tf_2\gamma_2,\,\ord_t(f_1\gamma_1+f_2\gamma_2)=n\right\}$$ with the natural action on jets and the natural morphism induced by the angular component. \\
Let $(\gamma_1,\gamma_2)\in\X_n^=(f_1\oplus f_2)$. Either $\ord_t(f_1\gamma_1)=\ord_t(f_2\gamma_2)=n$ with $\ac(f_1\gamma_1)+\ac(f_2\gamma_2)\neq0$ or $\ord_t(f_1\gamma_1)=\ord_t(f_2\gamma_2)<n$ with $\ac(f_1\gamma_1)+\ac(f_2\gamma_2)=0$. We now focus on this second case. \\
Let $h_i:M_i\rightarrow \mathbb R^{d_i}$ be as in Section \ref{subsubsect:monom}. \\
We lift the jets in $\X_n^=(f_1\oplus f_2)$ to jets in $\mathcal L_m(\mathbb R^{d_1+d_2})$ (second equality below) with $m\ge n$ big enough to apply the change of variables key lemma \ref{lem:keylemma} (fourth equality) in order to work locally with jets in $M_1\times M_2$ with origin in $\pring{E_I}\times\pring{E_J}$ (fifth equality) where $f_ih_i$ is a monomial times a unit.
\begin{align*}
(*)&=\left[(\gamma_1,\gamma_2)\in\mathcal L_n(\mathbb R^{d_1+d_2},0),\,\begin{array}{l}\ord_tf_1\gamma_1=\ord_tf_2\gamma_2<n,\\\ord_t(f_1\gamma_1+f_2\gamma_2)=n,\end{array}\right]\\
&=\left[(\gamma_1,\gamma_2)\in\mathcal L_m(\mathbb R^{d_1+d_2},0),\,\begin{array}{l}\ord_tf_1\gamma_1=\ord_tf_2\gamma_2<n,\\\ord_t(f_1\gamma_1+f_2\gamma_2)=n,\end{array}\right]\mathbb L^{(d_1+d_2)(n-m)} \\
&=\sum_{\substack{e\ge1\\e'\ge1}}\left[(\gamma_1,\gamma_2)\in\mathcal L_m(\mathbb R^{d_1+d_2},0),\,\begin{array}{l}\ord_tf_1\gamma_1=\ord_tf_2\gamma_2<n,\\\ord_t(f_1\gamma_1+f_2\gamma_2)=n,\end{array}\,\begin{array}{l}\gamma_1\in{h_1}_{*m}\pi_m\Delta_e,\\\gamma_2\in{h_2}_{*m}\pi_m\Delta_{e'}\end{array}\right]\mathbb L^{(d_1+d_2)(n-m)} \\
&=\sum_{\substack{e\ge1\\e'\ge1}}\left[(\gamma_1,\gamma_2)\in\mathcal L_m(M_1\times M_2,h_1^{-1}(0)\times h_2^{-1}(0)),\,\begin{array}{l}\ord_tf_1h_1\gamma_1=\ord_tf_2h_2\gamma_2<n,\\\ord_t(f_1h_1\gamma_1+f_2h_2\gamma_2)=n,\\\gamma_1\in\pi_m\Delta_e,\,\gamma_2\in\pi_m\Delta_{e'}\end{array}\right]\mathbb L^{(d_1+d_2)(n-m)-e-e'} \\
&=\sum_{\substack{e\ge1\\e'\ge1}}\sum_{\substack{I\neq\varnothing\\J\neq\varnothing}}\left[(\gamma_1,\gamma_2)\in\mathcal L_m(M_1\times M_2),\,\begin{array}{l}\gamma_1(0)\in h_1^{-1}(0)\cap\pring{E_I},\\\gamma_2(0)\in h_2^{-1}(0)\cap\pring{F_J}, \\\ord_tf_1h_1\gamma_1=\ord_tf_2h_2\gamma_2<n,\\\ord_t(f_1h_1\gamma_1+f_2h_2\gamma_2)=n,\\\gamma_1\in\pi_m\Delta_e,\,\gamma_2\in\pi_m\Delta_{e'}\end{array}\right]\mathbb L^{(d_1+d_2)(n-m)-e-e'}
\end{align*}

In a neighborhood of $\pring{E_I}$ we have $$f_1h_1(x)=u(x)\prod_{i\in I}x_i^{N_i(f_1)}$$ and in a neighborhood of $\pring{E_J}$ we have $$f_2h_2(y)=v(y)\prod_{j\in J}y_j^{N_j(f_2)}$$ where $u,v$ are units and $E_i:x_i=0,F_j:y_j=0$. Let $(\gamma_1,\gamma_2)\in\mathcal L(M_1\times M_2)$ satisfying the conditions of the last equality. Set $\gamma_{1i}(t)=t^{k_i}a_i(t)$ with $a_i(0)\neq 0$ and $\gamma_{2j}(t)=t^{l_j}b_j(t)$ with $b_j(0)\neq 0$. \\
Denote by $\omega=ord_t(f_1\gamma_1)=\ord_t(f_2\gamma_2)$. Then the coefficient of $t^\omega$ in $f_1h_1(\gamma_1(t))+f_2h_2(\gamma_2(t))$ is $$u(\gamma_1(0))\prod_{i\in I}a_{i0}^{N_i(f_1)}+v(\gamma_2(0))\prod_{j\in J}a_{j0}^{N_j(f_2)}$$
which is just $\ac(f_1h_1\gamma_1)+\ac(f_2h_2\gamma_2)$.

The coefficient of $t^{\omega+l}$ for $l=1,\ldots,n-\omega$ is of the form $$g_l(a_{il},b_{jl})+P_l$$
where $g_l$ is a non-zero linear form in $a_{il},b_{jl},i\in I,j\in J$ and where $P_l$ is a polynomial in $a_{ik},b_{jk},i\in I,j\in J,k<l$.

Since the coefficient of $t^\omega$ is zero, it brings the equation $\ac(f_1h_1\gamma_1)+\ac(f_2h_2\gamma_2)=0$. The coefficients of $t^{\omega+l}$ must be zero for $l=1,\ldots,n-\omega-1$, that brings the factor $\left(\mathbb L_{\AS}^{|I|+|J|-1}\right)^{n-\omega-1}=\mathbb L_{\AS}^{(|I|+|J|-1)(n-\omega-1)}$. The coefficient of $t^{n}$ must be non-zero and is the one which contributes to the angular component, hence it brings the factor $\mathbb L_{\AS}^{|I|+|J|-1}\cdot\mathbb1$. We have no condition for the other coefficients of $\gamma_{1i},i\in I$ and $\gamma_{2j},j\in J$, that brings the factor $\prod_{i\in I}\mathbb L_{\AS}^{m-k_i-n+\omega}=\mathbb L_{\AS}^{(m-n+\omega)|I|-\sum_Ik_i}$ and similarly $\mathbb L_{\AS}^{(m-n+\omega)|J|-\sum_Jl_j}$. We have no condition on the components of $\gamma_1$ (resp. $\gamma_2$) not indexed by $I$ (resp. $J$). This brings the factor $\mathbb L_{\AS}^{(d_1-|I|)m}$ and $\mathbb L_{\AS}^{(d_2-|J|)m}$. \\

Next
\begin{align*}
&\left[(\gamma_1,\gamma_2)\in\mathcal L_m(M_1\times M_2),\,\begin{array}{l}\gamma_1(0)\in h_1^{-1}(0)\cap\pring{E_I},\\\gamma_2(0)\in h_2^{-1}(0)\cap\pring{F_J}, \\\ord_tf_1h_1\gamma_1=\ord_tf_2h_2\gamma_2<n,\\\ord_t(f_1h_1\gamma_1+f_2h_2\gamma_2)=n,\\\gamma_1\in\pi_m\Delta_e,\,\gamma_2\in\pi_m\Delta_{e'}\end{array}\right]\\
&=\overline{[h_1^{-1}(0)\cap\pring{E_I}]}\overline{[h_2^{-1}(0)\cap\pring{F_J}]}\overline{[f_1h_1+f_2h_2\neq0]}\mathbb L^{\omega-n+(d_1+d_2)m}\sum_{\omega=1}^{n-1}\sum_{\substack{k_i\in\mathbb N,l_j\in\mathbb N\\\sum k_i(\nu_i(f_1)-1)=e\\\sum l_j(\nu_j(f_2)-1)=e'\\\sum k_iN_i(f_1)=\sum l_jN_j(f_2)=\omega}}\mathbb L^{-\sum k_i-\sum l_j}
\end{align*}
and we may similarly check that the RHS is also equal to
$$\sum_{\omega=1}^{n-1}\overline{\left[(\gamma_1,\gamma_2)\in\mathcal L_m(M_1\times M_2),\,\begin{array}{l}(\gamma_1(0),\gamma_2(0))\in(h_1^{-1}(0)\cap\pring{E_I})\times(h_2^{-1}(0)\cap\pring{F_J}),\\\ord_tf_1h_1\gamma_1=\ord_tf_2h_2\gamma_2=\omega,\\\ac(f_1h_1\gamma_1)+\ac(f_2h_2\gamma_2)=0,\\\gamma_1\in\pi_m\Delta_e,\,\gamma_2\in\pi_m\Delta_{e'}\end{array}\right]}\mathbb L^{\omega-n}$$
We now come back to arcs on $\mathbb R^{d_1+d_2}$.
\begin{align*}
(*)&=\sum_{\substack{e\ge1\\e'\ge1}}\sum_{\substack{I\neq\varnothing\\J\neq\varnothing}}\sum_{\omega=1}^{n-1}\overline{\left[(\gamma_1,\gamma_2)\in\mathcal L_m(M_1\times M_2),\,\begin{array}{l}(\gamma_1(0),\gamma_2(0))\in(h_1^{-1}(0)\cap\pring{E_I})\times(h_2^{-1}(0)\cap\pring{F_J}),\\\ord_tf_1h_1\gamma_1=\ord_tf_2h_2\gamma_2=\omega,\\\ac(f_1h_1\gamma_1)+\ac(f_2h_2\gamma_2)=0,\\\gamma_1\in\pi_m\Delta_e,\,\gamma_2\in\pi_m\Delta_{e'}\end{array}\right]}\mathbb L^{\omega-n+(d_1+d_2)(n-m)-e-e'} \\
&=\sum_{\omega=1}^{n-1}\overline{\left[(\gamma_1,\gamma_2)\in\mathcal L_m(\mathbb R^{d_1+d_2},0),\,\begin{array}{l}\ord_tf_1\gamma_1=\ord_tf_2\gamma_2=\omega\\\ac(f_1\gamma_1)+\ac(f_2\gamma_2)=0\end{array}\right]}\mathbb L^{\omega-n+(d_1+d_2)(n-m)}\\
&=\sum_{\omega=1}^{n-1}\overline{\left[(\gamma_1,\gamma_2)\in\mathcal L_\omega(\mathbb R^{d_1+d_2},0),\,\begin{array}{l}\ord_tf_1\gamma_1=\ord_tf_2\gamma_2=\omega\\\ac(f_1\gamma_1)+\ac(f_2\gamma_2)=0\end{array}\right]}\mathbb L^{\omega-n+(d_1+d_2)(n-m)-(d_1+d_2)(\omega-m)}\\
&=\sum_{\omega=1}^{n-1}\overline{\left[(\gamma_1,\gamma_2)\in\X_\omega(f_1)\times\X_\omega(f_2),\,\ac f_1\gamma_1+\ac f_2\gamma_2=0\right]}\mathbb L^{(d_1+d_2-1)(n-\omega)}
\end{align*}
Finally we get
\begin{align*}
\left[\X_n^=(f_1\oplus f_2)\right]&=\left[(\gamma_1,\gamma_2)\in\X_n(f_1)\times\X_n(f_2),\,\ac f_1\gamma_1+\ac f_2\gamma_2\neq0\right] \\
&\quad\quad+ \sum_{\omega=1}^{n-1}\overline{\left[(\gamma_1,\gamma_2)\in\X_\omega(f_1)\times\X_\omega(f_2),\,\ac f_1\gamma_1+\ac f_2\gamma_2=0\right]}\mathbb L^{(d_1+d_2-1)(n-\omega)} \\
&=\left[(\gamma_1,\gamma_2)\in\X_n(f_1)\times\X_n(f_2),\,\ac f_1\gamma_1+\ac f_2\gamma_2\neq0\right] \\
&\quad\quad-\left[(\gamma_1,\gamma_2)\in\X_n(f_1)\times\X_n(f_2),\,\ac f_1\gamma_1+\ac f_2\gamma_2=0\right] \\
&\quad\quad+  \sum_{\omega=1}^{n}\overline{\left[(\gamma_1,\gamma_2)\in\X_\omega(f_1)\times\X_\omega(f_2),\,\ac f_1\gamma_1+\ac f_2\gamma_2=0\right]}\mathbb L^{(d_1+d_2-1)(n-\omega)} \\
&=-\left[\X_n(f_1)\right]*\left[\X_n(f_2)\right]+\sum_{\omega=1}^{n}\overline{\left[(\gamma_1,\gamma_2)\in\X_\omega(f_1)\times\X_\omega(f_2),\,\ac f_1\gamma_1+\ac f_2\gamma_2=0\right]}\mathbb L^{(d_1+d_2-1)(n-\omega)}
\end{align*}
This ends the second case. \\

Therefore the computations of the beginning of the proof give
\begin{align*}
\left[\X_n(f_1\oplus f_2)\right]\mathbb L^{-n(d_1+d_2)}&=\left(\left[\X_n^{\neq}(f_1\oplus f_2)\right]+\left[\X_n^{=}(f_1\oplus f_2)\right]\right)\mathbb L^{-n(d_1+d_2)}\\
&=\left(1-\sum_{\omega=1}^n\overline{[\X_\omega(f_2)]}\mathbb L^{-\omega d_2}\right)[\X_n(f_1)]\mathbb L^{-nd_1}\\
&+\left(1-\sum_{\omega=1}^n\overline{[\X_\omega(f_1)]}\mathbb L^{-\omega d_1}\right)[\X_n(f_2)]\mathbb L^{-nd_2}\\
&-\left[\X_n(f_1)\right]\mathbb L^{-nd_1}*\left[\X_n(f_2)\right]\mathbb L^{-nd_2}\\
&+\sum_{\omega=1}^n\overline{\left[(\gamma_1,\gamma_2)\in\X_{\omega}(f_1)\times\X_{\omega}(f_2),\,\ac(f_1\gamma_1)+\ac(f_2\gamma_2)=0\right]}\mathbb L^{-\omega(d_1+d_2)}\mathbb L^{\omega-n}
\end{align*}

Using again a monomialization, we get that
\begin{align*}
&\overline{\left[(\gamma_1,\gamma_2)\in\mathcal L_n(\mathbb R^{d_1+d_2}),\,\ord_t(f_1\gamma_1+f_2\gamma_2)>n\right]}\cdot\mathbb1\\
&\quad\quad\quad=\overline{\left[\ord_tf_1>n,\ord_tf_2>n\right]}\cdot\mathbb1+\sum_{\omega=1}^n\overline{\left[\ord_tf_1=\ord_tf_2=\omega,\,\ord_t(f_1+f_2)>n\right]}\cdot\mathbb1\\
&\quad\quad\quad=\overline{\left[\ord_tf_1>n,\ord_tf_2>n\right]}\cdot\mathbb1+\sum_{\omega=1}^n\overline{\left[(\gamma_1,\gamma_2)\in\X_{\omega}(f_1)\times\X_{\omega}(f_2),\,\ac(f_1\gamma_1)+\ac(f_2\gamma_2)=0\right]}\mathbb L^{(n-\omega)(d_1+d_2-1)}
\end{align*}

This allows us to conclude as follows.
\begin{align*}
&\left[\Y_n(f_1\oplus f_2)\right]\mathbb L^{-n(d_1+d_2)}\\
&\quad\quad=\left[\X_n(f_1\oplus f_2)\right]\mathbb L^{-n(d_1+d_2)}-\overline{\left[(\gamma_1,\gamma_2)\in\mathcal L_n(\mathbb R^{d_1+d_2}),\,\ord_t(f_1\gamma_1+f_2\gamma_2)>n\right]}\mathbb L^{-n(d_1+d_2)}\\
&\quad\quad=\left(1-\sum_{\omega=1}^n\overline{[\X_\omega(f_2)]}\mathbb L^{-\omega d_2}\right)[\X_n(f_1)]\mathbb L^{-nd_1}\\
&\quad\quad\quad\quad+\left(1-\sum_{\omega=1}^n\overline{[\X_\omega(f_1)]}\mathbb L^{-\omega d_1}\right)[\X_n(f_2)]\mathbb L^{-nd_2}\\
&\quad\quad\quad\quad-\left[\X_n(f_1)\right]\mathbb L^{-nd_1}*\left[\X_n(f_2)\right]\mathbb L^{-nd_2}\\
&\quad\quad\quad\quad-\overline{\left[\ord_tf_1>n\right]}\mathbb L^{-nd_1}\overline{\left[\ord_tf_2>n\right]}\mathbb L^{-nd_2}\\
&\quad\quad=-\left(\left[\Y_n(f_1)\right]\mathbb L^{-nd_1}\right)*\left(\left[\Y_n(f_2)\right]\mathbb L^{-nd_2}\right)
\end{align*}
\end{proof}

We recover the convolution \cite[Theorem 2.3]{KP03} thanks to the following lemma.
\begin{lemma}\label{lem:TSKP03}
Let $\varepsilon\in\{>,<\}$ then $$\chi_c(F^\varepsilon(x*y))=-\chi_c(F^\varepsilon x)\chi_c(F^\varepsilon y)$$
\end{lemma}

We need the following lemma in order to prove the previous one.
\begin{lemma}\label{lem:semialgTF}
Let $X$ be an $\AS$-set endowed with an $\AS$ action of $\mathbb R^*$ and $\varphi:X\rightarrow\mathbb R^*$ be an $\AS$-map such that $\varphi(\lambda\cdot x)=\lambda^n\varphi(x)$. Then $\varphi$ is a trivial semialgebraic fibration over $\mathbb R_{>0}$ and over $\mathbb R_{<0}$ (or over $\mathbb R^*$ if $n$ is odd).
\end{lemma}
\begin{proof}
Indeed, the following diagram commutes (for the case $>0$)
$$\xymatrix{\varphi^{-1}(\mathbb R_{>0}) \ar[rr]^{\Psi} \ar[rd]_\varphi & & \varphi^{-1}(1)\times\mathbb R_{>0} \ar[ld]^{\pr_{\mathbb R_{>0}}} \\ & \mathbb R_{>0} &}$$
where $\Psi(x)=(\varphi(x)^{-\frac{1}{n}}\cdot x,\varphi(x))$ and $\Psi^{-1}(x,\lambda)=\lambda^{\frac{1}{n}}\cdot x$
\end{proof}

\begin{proof}[Proof of Lemma \ref{lem:TSKP03}]
Assume that $\varepsilon =\,>$. Then 
\begin{align*}
\chi_c\left(F^{>}([\varphi_1:X_1\rightarrow\mathbb R^*]*[\varphi_2:X_2\rightarrow\mathbb R^*])\right)&=-\chi_c\left((\varphi_1+\varphi_2)>0\right)+\chi_c\left((\varphi_1+\varphi_2=0)\times\mathbb R_{>0}\right) \\
&=-\chi_c\left((\varphi_1+\varphi_2)>0\right)-\chi_c\left((\varphi_1+\varphi_2=0)\right)
\intertext{Where the last equality comes from the fact that $\chi_c(\mathbb R_{>0})=-1$. Thus}
\chi_c\left(F^{>}([\varphi_1:X_1\rightarrow\mathbb R^*]*[\varphi_2:X_2\rightarrow\mathbb R^*])\right)&=-\chi_c\left((\varphi_1+\varphi_2)\ge0\right) \\
&=-\chi_c\left(\varphi_1\ge0,\varphi_2\ge0\right)\\
&\quad\quad\quad+\chi_c\left((\varphi_1+\varphi_2)\ge0,\varphi_1<0\right)+\chi_c\left((\varphi_1+\varphi_2)\ge0,\varphi_2<0\right)
\intertext{Since $\varphi_i$ is trivial over $\mathbb R_{>0}$ (resp. $\mathbb R_{<0}$) and $\chi_c(a+b\ge0,a<0)=0$, we get}
\chi_c\left(F^{>}([\varphi_1:X_1\rightarrow\mathbb R^*]*[\varphi_2:X_2\rightarrow\mathbb R^*])\right)&=-\chi_c\left(\varphi_1\ge0,\varphi_2\ge0\right)\\
&=-\chi_c\left(\varphi_1>0,\varphi_2>0\right)\\
&\quad\quad\quad-\chi_c\left(\varphi_1\ge0,\varphi_2=0\right)-\chi_c\left(\varphi_1=0,\varphi_2\ge0\right)
\intertext{Since $\varphi_i$ is trivial over $\mathbb R_{>0}$ (resp. $\mathbb R_{<0}$) and $\chi_c(a\ge0,b=0)=0$, we finally have}
\chi_c\left(F^{>}([\varphi_1:X_1\rightarrow\mathbb R^*]*[\varphi_2:X_2\rightarrow\mathbb R^*])\right)&=-\chi_c\left(\varphi_1>0,\varphi_2>0\right)\\
&=-\chi_c\left(F^{>}[\varphi_1:X_1\rightarrow\mathbb R^*]\right)\chi_c\left(F^{>}[\varphi_2:X_2\rightarrow\mathbb R^*]\right)
\end{align*}
\end{proof}

\begin{cor}[{\cite[Theorem 2.3]{KP03}}]
Let $\varepsilon\in\{>,<\}$ then $$\tilde Z_{f\oplus g}^{\chi_c,\varepsilon}(T)=\tilde Z_{f}^{\chi_c,\varepsilon}(T)\odot\tilde Z_{g}^{\chi_c,\varepsilon}(T)$$ where the product $\odot$ is the Hadamard product which consists in applying the classical product of $\mathbb Z$ coefficientwise.
\end{cor}

We showed in the proof of \ref{lem:semialgTF} that $\chi_c(F^>(a))=-\chi_c(F^+(a))$. In the same way we may prove that $\chi_c(F^<(a))=-\chi_c(F^-(a))$. From these facts we derive the following lemma.

\begin{lemma}\label{lem:TSchiFplus}
Let $\varepsilon\in\{+,-\}$ then $$\chi_c(F^\varepsilon(x*y))=\chi_c(F^\varepsilon x)\chi_c(F^\varepsilon y)$$
\end{lemma}

\begin{cor}
Let $\varepsilon\in\{+,-\}$ then $$\tilde Z_{f\oplus g}^{\chi_c,\varepsilon}(T)=-\tilde Z_{f}^{\chi_c,\varepsilon}(T)\odot\tilde Z_{g}^{\chi_c,\varepsilon}(T)$$ where the product $\odot$ is the Hadamard product which consists in applying the classical product of $\mathbb Z$ coefficientwise.
\end{cor}

\section{Arc-analytic equivalence}\label{sect:bne}
T.-C. Kuo \cite{Kuo85} defined the notion of blow-analytic equivalence for real analytic function germs: two germs are blow-analytically equivalent if we can get one from the other by composing with a homeomorphism which is blow-analytic and such that the inverse is also blow-analytic. In order to prove that this is an equivalence relation, he gave a characterization in terms of real modifications \cite[Proposition 2]{Kuo85}. Similarly, G. Fichou \cite[Definition 4.1]{Fic05} defined the notion of blow-Nash equivalence for Nash function germs, which is an algebraic version of the blow-analytic equivalence of T.-C. Kuo. G. Fichou proved that his zeta functions are invariants of this relation. However it is not clear that this notion is an equivalence relation.

The definition of blow-Nash equivalence evolved, and we now use the following one \cite{Fic06,Fic08,FF}: two Nash function germs are blow-Nash equivalent if they are equivalent via a blow-Nash isomorphism in the sense of \cite[Definition 1.1]{Fic05-bis}. The assumption ``via a blow-Nash isomorphism'' is needed to ensure that the zeta functions of G. Fichou are invariants of this relation, but with this assumption, it is still not clear whether it is an equivalence relation.

In this section, we introduce a new relation on Nash function germs, the arc-analytic equivalence which is an equivalence relation. Moreover it coincides with the current version of the blow-Nash equivalence. To introduce it, we do not need Nash modifications. Finally, our zeta function is an invariant of this relation.

\begin{defn}
A semialgebraic function $f:V\rightarrow\mathbb R$ defined on an algebraic set $V$ is said to be blow-Nash if there exists $\sigma:\tilde V\rightarrow V$ a finite sequence of (algebraic) blowings-up with non-singular centers such that $\tilde V$ is non-singular and $f\circ\sigma:\tilde V\rightarrow\mathbb R$ is Nash.
\end{defn}

\begin{defn}
A real analytic arc on $\mathbb R^d$ is a real analytic map $\gamma:(-\varepsilon,\varepsilon)\rightarrow\mathbb R^d$.
\end{defn}

We recall the following useful result of Bierstone-Milman.
\begin{thm}[{\cite[Theorem 1.1]{BM90}}]
A semialgebraic function $f:U\rightarrow\mathbb R$ defined on a non-singular real algebraic set is blow-Nash if and only if it sends real analytic arcs to real analytic arcs by composition. 
\end{thm}

\begin{defn}
A map that sends real analytic arcs to real analytic arcs by composition is called \emph{arc-analytic}.
\end{defn}

\begin{defn}\label{defn:blowNash}
Two germs $f,g:(\mathbb R^d,0)\rightarrow(\mathbb R^d,0)$ are said to be arc-analytic equivalent if there exists a semialgebraic homeomorphism $h:(\mathbb R^d,0)\rightarrow(\mathbb R^d,0)$ satisfying $f=g\circ h$ such that $h$ is arc-analytic and such that there exists $c>0$ with $|\det\d h|>c$ where $\d h$ is defined.
\end{defn}

\begin{rem}
Let $f,g:(\mathbb R^d,0)\rightarrow(\mathbb R^d,0)$ and $h:(\mathbb R^d,0)\rightarrow(\mathbb R^d,0)$ be as in Definition \ref{defn:blowNash}. Then there exists $\varphi:M\rightarrow\mathbb R^d$ a finite sequence of blowings-up with non-singular centers such that $\tilde\varphi=h\circ\varphi$ is Nash:
$$
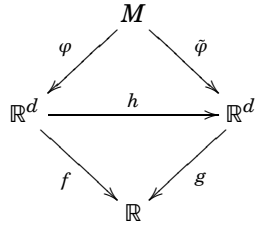
\xymatrix{
& M \ar[dl]_{\varphi} \ar[dr]^{\tilde\varphi} & \\
\mathbb R^d \ar[rr]^h \ar[dr]_f & & \mathbb R^d \ar[dl]^g \\
& \mathbb R &
}$$
Notice that then $\tilde\varphi$ is proper generically one-to-one Nash. Moreover, by \cite[Corollary 4.17]{jbc1}, $\tilde\varphi$ induces a bijection between arcs on $M$ and arcs on $\mathbb R^d$ not entirely included in some nowhere dense subset of $\mathbb R^d$.
\end{rem}

\begin{prop}
Arc-analytic equivalence is an equivalence relation.
\end{prop}
\begin{proof}
The reflexivity is obvious, the symmetry comes from \cite[Theorem 3.5]{jbc1}. Thus it suffices to prove the transitivity. We have the following diagram $$
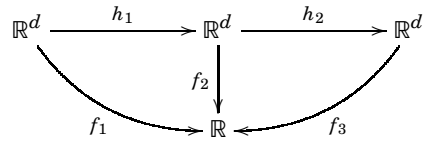
\xymatrix{\mathbb R^d \ar[rr]^{h_1} \ar@/_1pc/[rrd]_{f_1} & & \mathbb R^d \ar[rr]^{h_2} \ar[d]_{f_2} & & \mathbb R^d \ar@/^1pc/[lld]^{f_3} \\ &&\mathbb R&&}$$ where the $f_i$ are Nash germs and the $h_i$ are as in Definition \ref{defn:blowNash}. Obviously there exists $c>0$ such that $|\det\d(h_2\circ h_1)|>c$. The composition $h_2\circ h_1$ is obviously semialgebraic and arc-analytic (as the composition of such maps).
\end{proof}

\begin{prop}\label{prop:BNFic1}
Two Nash germs which are blow-Nash equivalent in the sense of \cite[Definition 4.1]{Fic05} are arc-analytic equivalent.
\end{prop}
\begin{proof}
Assume that $f_1$ and $f_2$ are blow-Nash equivalent in the sense of \cite{Fic05}. Then we have
$$
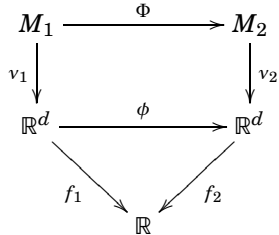
\xymatrix{M_1 \ar[rr]^\Phi \ar[d]_{\nu_1} & & M_2 \ar[d]^{\nu_2} \\ \mathbb R^d \ar[rr]^\phi \ar[rd]_{f_1} & & \mathbb R^d \ar[ld]^{f_2} \\ & \mathbb R &}$$
with $\phi$ a semialgebraic homeomorphism, $\nu_i$ two proper birational algebraic maps and $\Phi$ a Nash isomorphism which preserves the multiplicities of the jacobian determinants of $\nu_1$ and $\nu_2$. \\
Since we may lift analytic arcs by $\nu_1$, $\phi$ is arc-analytic. By the chain rule, since $\Phi$ preserves the multiplicities of the jacobian determinants of $\nu_1$ and $\nu_2$, we deduce that $|\det\d\phi|>c$ for some $c>0$.
\end{proof}

\begin{prop}\label{prop:BNFic2}
Two Nash germs are arc-analytic equivalent if and only if they are blow-Nash equivalent via a blow-Nash isomorphism in the sense of \cite[Definition 1.1]{Fic05-bis}.
\end{prop}
\begin{proof}
Assume that $f_1$ and $f_2$ are blow-Nash equivalent via a blow-Nash isomorphism in the sense of \cite{Fic05-bis}. Then we have
$$
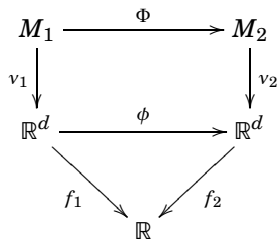
\xymatrix{M_1 \ar[rr]^\Phi \ar[d]_{\nu_1} & & M_2 \ar[d]^{\nu_2} \\ \mathbb R^d \ar[rr]^\phi \ar[rd]_{f_1} & & \mathbb R^d \ar[ld]^{f_2} \\ & \mathbb R &}$$
with $\phi$ a semialgebraic homeomorphism, $\nu_i$ two Nash modifications\footnote{A Nash modification is a proper surjective Nash map whose complexification is proper and bimeromorphic.} and $\Phi$ a Nash isomorphism which preserves the multiplicities of the jacobian determinants of $\nu_1$ and $\nu_2$. \\
Since we may lift analytic arcs by $\nu_1$, $\phi$ is arc-analytic. 
By the chain rule, since $\Phi$ preserves the multiplicities of the jacobian determinants of $\nu_1$ and $\nu_2$, we deduce that $|\det\d\phi|>c$ for some $c>0$.

Assume that $f_1$ and $f_2$ are arc-analytic equivalent via $h$: $f_1=f_2\circ h$ with $h$ semialgebraic and arc-analytic satisfying $|\det\d h|>c>0$. Since $h$ and $h^{-1}$ are blow-Nash there exist $\nu_1:M_1\rightarrow\mathbb R^d$ and $\nu_2:M_2\rightarrow\mathbb R^d$ two finite sequences of blowings-up with non-singular centers such that $\alpha_1=h\circ\nu_1$ and $\alpha_2=h^{-1}\circ\nu_2$ are Nash. Let $N_1$ (resp. $N_2$) be the fiber product of $\alpha_1$ and $\nu_2$ (resp. $\alpha_2$ and $\nu_1$). Then $N_1=N_2=N$ in $M_1\times M_2$. Notice that $\pi_i$ and $N_i$ are Nash. Let $\sigma:\tilde N\rightarrow N$ be a resolution of singularities (for Nash spaces, see \cite[p234]{BM97}). Then for $i=1,2$, $\nu_i\pi_i\sigma$ is a Nash modification.
$$
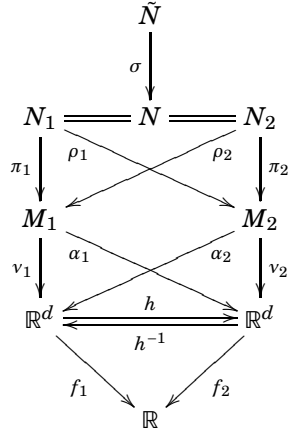
\xymatrix{
& \tilde N \ar[d]_\sigma & \\
N_1 \ar[d]_{\pi_1} \ar[rrd]_<<<<{\rho_1} \ar@{=}[r] & N \ar@{=}[r] & N_2 \ar[d]^{\pi_2} \ar[lld]^<<<<{\rho_2} \\
M_1 \ar[rrd]_<<<<{\alpha_1} \ar[d]_{\nu_1} & & M_2 \ar[lld]^<<<<{\alpha_2} \ar[d]^{\nu_2} \\
\mathbb R^d \ar@<+.3ex>[rr]^h \ar[rd]_{f_1} & & \mathbb R^d \ar[ld]^{f_2} \ar@<+.3ex>[ll]^{h^{-1}}  \\
& \mathbb R &
}$$
\end{proof}

\begin{cor}\label{cor:BNisER}
Blow-Nash equivalence via a blow-Nash isomorphism in the sense of \cite[Definition 1.1]{Fic05-bis} is an equivalence relation.
\end{cor}

\begin{thm}
If two Nash germs $f,g:(\mathbb R^d,0)\rightarrow(\mathbb R,0)$ are arc-analytic equivalent then $Z_f(T)=Z_g(T)$.
\end{thm}
\begin{proof}
We have the following diagram
$$
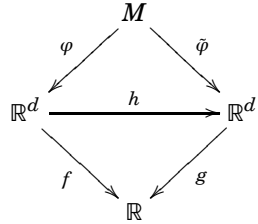
\xymatrix{
& M \ar[dl]_\varphi \ar[dr]^{\tilde\varphi} & \\
\mathbb R^d \ar[rr]^h \ar[rd]_{f} & & \mathbb R^d \ar[ld]^{g} \\
& \mathbb R &
}$$
where $h$ is a semialgebraic homeomorphism, $\varphi$ a finite sequence of blowings-up with non-singular centers and $\tilde\varphi$ a Nash map. \\

Notice that in the statement of Theorem \ref{thm:rat} we only need that $h$ is proper generically 1-to-1 Nash and not merely birational. Therefore we may also apply this theorem to $\tilde\varphi$ in order to compute $Z_g(T)$. \\

Up to adding more blowings-up we may assume that $f\circ\varphi$, $g\circ\tilde\varphi=g\circ h\circ\varphi$, $\Jac\varphi$ and $\Jac\tilde\varphi$ are simultaneously normal crossings. Denote by $(E_i)_{i\in A}$ the irreducible components of the zero set of $f\circ\varphi=g\circ\tilde\varphi$. Set $f\circ\varphi=\sum_{i\in A}N_iE_i$, $g\circ\tilde\varphi=\sum_{i\in A}\tilde N_iE_i$, $\Jac\varphi=\sum_{i\in A}(\nu_i-1)E_i$ and $\Jac\tilde\varphi=\sum_{i\in A}(\tilde\nu_i-1)E_i$. \\
By \cite[Lemma 4.15]{jbc1}, $\forall i\in A,\,\nu_i=\tilde\nu_i$. And since $f\circ\varphi=g\circ\tilde\varphi$, we have $\forall i\in A,\,N_i=\tilde N_i$ and that $U_I$ is well-defined and doesn't depend on $\varphi$ or $\tilde\varphi$. \\

Thus by Theorem \ref{thm:rat} $$Z_f(T)=Z_g(T)=\sum_{\varnothing\neq I\subset A}\left[U_I\cap(h\circ p_I)^{-1}(0)\right]\prod_{i\in I}\frac{\mathbb L^{-\nu_i}T^{N_i}}{1-\mathbb L^{-\nu_i}T^{N_i}}$$
\end{proof}

\begin{rem}
Particularly, by Proposition \ref{prop:BNFic2}, we recover \cite[Proposition 2.6]{Fic05-bis}.
\end{rem}

By proposition \ref{prop:BNFic1}, \cite[Theorem 4.3]{Fic05} works as it is in our settings (see also \cite[Theorem 1.5]{Fic05-bis} and \cite[Theorem 1]{Kuo85}).
\begin{thm}\label{thm:triviality}
Let $F:(\mathbb R^d,0)\times(0,1)^k\rightarrow(\mathbb R,0)$ be a Nash function such that $\forall t\in(0,1)^k$, $F(\cdot,t):(\mathbb R^d,0)\rightarrow(\mathbb R,0)$ has an isolated singularity at $0$ and there exists an algebraic proper birational map $\sigma:M\rightarrow\mathbb R^d\times(0,1)^k$ such that $F\circ\sigma$ has only normal crossings. Then the elements of the family $F_t(\cdot)=F(\cdot,t)$ represent a finite number of arc-analytic equivalence classes.
\end{thm}

In the same way, we recover a version of \cite[Proposition 4.17]{Fic05}.
\begin{prop}
Let $F:(\mathbb R^d,0)\times(0,1)^k\rightarrow(\mathbb R,0)$ be a Nash function such that $\forall t\in(0,1)^k$, $F(\cdot,t):(\mathbb R^d,0)\rightarrow(\mathbb R,0)$ has an isolated singularity at $0$ and there exists an algebraic proper birational map $\sigma:M\rightarrow\mathbb R^d$ such that $F\circ(\sigma,\id_{(0,1)^k})$ has only normal crossings. Then the elements of the family $F_t(\cdot)=F(\cdot,t)$ represent a unique arc-analytic equivalence class.
\end{prop}

Again, the following corollary is just a version of \cite[Corollary 4.5]{Fic05}.
\begin{cor}\label{cor:BNTrivFam}
Let $F:(\mathbb R^d,0)\times(0,1)\rightarrow(\mathbb R,0)$ be a Nash function such that $F(\cdot,t):(\mathbb R^d,0)\rightarrow(\mathbb R,0)$ are weighted homogeneous polynomials with the same weights and have an isolated singularity at $0$. Then the elements of the family $F_t(\cdot)=F(\cdot,t)$ represent a unique arc-analytic equivalence class.
\end{cor}

\begin{eg}
The Whitney family \cite[Example 13.1]{Whi65} $f_t(x,y)=xy(y-x)(y-tx)$, $t\in(0,1)$, has only one arc-analytic equivalence class.
\end{eg}

\section{Classification of Brieskorn polynomials}
In \cite{KP03}, S. Koike and A. Parusiński gave a complete classification of the Brieskorn polynomials in two variables up to blow-analytic equivalence using their zeta functions and the Fukui invariants. In \cite{Fic05}, G. Fichou classified the Brieskorn polynomials in three variables up to blow-Nash equivalence (and thus up to arc-analytic equivalence) thanks to his zeta functions. 
In this section we use the convolution formula to prove that two arc-analytic equivalent Brieskorn polynomials share the same exponents.

\begin{defn}
A Brieskorn polynomial is a polynomial of the following form $$f(x)=\sum_{i=1}^d\varepsilon_ix^{k_i}$$ where $\varepsilon_i\in\{\pm1\}$ and $k_i\ge1$.
\end{defn}

\begin{rem}
We will assume that $k_i\ge2$ since otherwise the polynomial is non-singular. Up to reordering the variables, we will always assume that $2\le k_1\le k_2\le\cdots\le k_d$.
\end{rem}

\begin{prop}\label{prop:recoverweights}
Let $$f(x)=\sum_{i=1}^d\varepsilon_ix^{k_i}$$ be a Brieskorn polynomial. We may recover $(k_1,\ldots,k_d)$ from $d$ and $Z_f(T)$.
\end{prop}

\begin{cor}
Let $$f(x)=\sum_{i=1}^d\varepsilon_ix^{k_i}\quad\quad\quad\quad\text{ and }\quad\quad\quad\quad g(x)=\sum_{i=1}^d\eta_ix^{l_i}$$ be two Brieskorn polynomials. If $f$ and $g$ are blow-Nash equivalent then $\forall i=1,\ldots,d$ we have $k_i=l_i$.
\end{cor}

Remember from Example \ref{eg:monomial} that
$$\hspace{-1.5cm}\tilde Z_{\varepsilon_ix_i^{k_i}}(T) = -T-\cdots-T^{k_i-1}-\left(\mathbb1-[\varepsilon_ix_i^{k_i}]\right)\mathbb L^{-1}T^{k_i}-\mathbb L^{-1}T^{k_i+1}-\cdots-\mathbb L^{-1}T^{2k_i-1}-\left(\mathbb1-[\varepsilon_ix_i^{k_i}]\right)\mathbb L^{-2}T^{2k_i}-\mathbb L^{-2}T^{2k_i+1}-\cdots$$

Notice that the coefficients of terms whose degrees are not multiples of $k_i$ are of the form $-\mathbb L^{-\alpha}$ where $\alpha$ is the integral part of the degree divided by $k_i$.

Denote by $a_n\in\M$ the coefficients of the modified zeta function of $f$ so that $$\tilde Z_f(T)=\sum_{n\ge1}a_n T^n\in \M[[T]]$$
Using the convolution formula and that $\mathbb L^{-\alpha}*\mathbb L^{-\beta}=\mathbb L^{-(\alpha+\beta)}$ ($\M_{\AS}$-bilinearity of the convolution product) we deduce from the previous remark that if $k_i$ doesn't divide $n$ for all $i$ then $$a_n=-\mathbb L^{-\sum_{i=1}^d\left\lfloor\frac{n}{k_i}\right\rfloor}$$

Moreover, we may recover the degree $\sum_{i=1}^d\left\lfloor\frac{n}{k_i}\right\rfloor$ using the forgetful morphism $\overline{\vphantom{1em}\ \cdot\ }:\M\rightarrow\M_{\AS}$ and the natural extension of the virtual Poincaré polynomial to $\M_{\AS}$, $\beta:\M_{\AS}\rightarrow\mathbb Z[u,u^{-1}]$. Indeed,
$$\beta\left(\overline{a_n}\right)=\beta\left(\overline{-\mathbb L^{-\sum_{i=1}^d\left\lfloor\frac{n}{k_i}\right\rfloor}}\right)=\beta\left(-\mathbb L_{\AS}^{-\sum_{i=1}^d\left\lfloor\frac{n}{k_i}\right\rfloor}(\mathbb L_{\AS}-1)\right)=-u^{-\sum_{i=1}^d\left\lfloor\frac{n}{k_i}\right\rfloor}(u-1)$$
and thus 
$$\sum_{i=1}^d\left\lfloor\frac{n}{k_i}\right\rfloor=-\deg\beta(\overline{a_n})+1$$

The idea of the proof of Proposition \ref{prop:recoverweights} is to use the previous fact to reduce to a combinatorial problem.

\begin{lemma}\label{lem:lemFE}
Fix $m\in\mathbb N_{>0}$. Let $x=\displaystyle\sum_{i=1}^m\frac{1}{l_i}$ with $l_i\in\mathbb N_{>0}$. Then there exists a finite number of $m$-tuples $(l_i')_{i=1,\ldots,m}$ such that $x=\displaystyle\sum_{i=1}^m\frac{1}{l_i'}$.
\end{lemma}
\begin{proof}
Assume that the statement is true for some $m\in\mathbb N_{>0}$ and let $x=\displaystyle\sum_{i=1}^{m+1}\frac{1}{l_i}$. We may assume that $l_1\le\cdots\le l_{m+1}$. Then $l_1\le\frac{m+1}{x}$. Then there is a finite number of choices for $l_1$ and for every choice $x-\frac{1}{l_1}=\displaystyle\sum_{i=2}^{m+1}\frac{1}{l_i}$ admits a finite number of expressions of this form.
\end{proof}

\begin{proof}[Proof of Proposition \ref{prop:recoverweights}]
Let $$\tilde Z_f(T)=\sum_{n\ge1}a_n T^n\in\M[[T]]$$ Denote by $\mathcal P$ the set of primes. For $p\in\mathcal P$ big enough, $p$ is not a multiple of a $k_i$. Thus $$\lim_{p\in\mathcal P}\frac{-\deg\left(\beta(\overline{a_p})\right)+1}{p}=\lim_{p\in\mathcal P}\frac{\sum_{i=1}^d\left\lfloor\frac{p}{k_i}\right\rfloor}{p}=\sum_{i=1}^d\frac{1}{k_i}$$
We deduce from the previous computation and Lemma \ref{lem:lemFE} that we may derive from $\tilde Z_f(T)$ (or equivalently from $Z_f(T)$) an integer $K$ such that $k_1\le\cdots\le k_d\le K$.

Denote by $\mathcal P'$ the set of primes lower or equal to $K$. For $p\in \mathcal P'$, we denote by $\gamma_p$ the greatest exponent such that $p^{\gamma_p}\le K$.

Set $$Q=\left\{\prod_{p\in\mathcal P'}p^{\alpha_p},\,0\le\alpha_p\le\gamma_p\right\}$$ so that $\{k_1,\ldots,k_d\}\subset Q$. Up to adding elements in $Q$, we may assume that $2,3,5,7\in\mathcal P'$ and that $\gamma_2\ge3,\gamma_3\ge2,\gamma_5\ge1,\gamma_7\ge1$.

For $q\in Q$, set $\mult(q)=\#\{k_i,\,k_i=q\}$. Thus our goal is to compute $\mult q,\,q\in Q$, from $Z _f(T)$. \\

The main idea of the proof consists, for a number $q\in Q$, to use the Chinese remainder theorem to find $n$ such that $n-1$ and $n+1$ are not multiple of a term in $Q\setminus\{1\}$ and such that the only factors of $n$ that are in $Q$ are the factors of $q$. The first condition ensures that no exponent divides $n-1$ and $n+1$ so that we can recover the degrees of $a_{n-1}$ and $a_{n+1}$ as previously. The second condition ensures that the difference of these degrees is exactly $$\sum_{k|q}\mult k$$
This allows us to compute $\mult q$ recursively by increasing the number of factors in $q$. This is exactly the first item below.

Unfortunately this method won't work when $2\nmid q$ or $3\nmid q$ since then $3$ or $2$ may divide $n-1$ or $n+1$ (whereas $2$ and $3$ may appear as exponents in the polynomial). Thus we have to manage these cases separately in a similar, but more sophisticated, way. \\

We first notice that $\mult(1)=0$.

\begin{itemize}
\item Equations involving $\mult q$ for $q\in Q$ satisfying $6|q$. \\
Let $\alpha_2,\alpha_3$ be such that $1\le\alpha_2\le\gamma_2$ and $1\le\alpha_3\le\gamma_3$ and for each $p\in\mathcal P'\setminus\{2,3\}$, let $\alpha_p$ be such that $0\le\alpha_p\le\gamma_p$.

The Chinese remainder theorem ensures the existence of $n$ such that $$\left\{\begin{array}{ll}n\equiv p^{\alpha_p}\mod p^{\alpha_p+1} & \text{ if }\alpha_p>0\\n\equiv 2 \mod p&\text{ otherwise}\end{array}\right.$$

Thus, for $p\in \mathcal P'$, $p$ doesn't divide $n-1$ and $n+1$. This ensures that, for $q\in Q$, $n-1$ and $n+1$ are not multiple of $q$ and particularly of an exponent $k_i$. So $a_{n-1}=-\mathbb L^{-\sum_{i=1}^d\left\lfloor\frac{n-1}{k_i}\right\rfloor}$ and $a_{n+1}=-\mathbb L^{-\sum_{i=1}^d\left\lfloor\frac{n+1}{k_i}\right\rfloor}$.

Moreover the elements of $Q$ which divide $n$ are exactly those of the form $\displaystyle\prod_{p\in\mathcal P'}p^{\beta_p}$ with $0\le\beta_p\le\alpha_p$.

Therefore $$-\deg\beta(\overline{a_{n+1}})+\deg\beta(\overline{a_{n-1}})=\sum_{0\le\beta_p\le\alpha_p}\mult\left(\prod_{p\in\mathcal P'}p^{\beta_p}\right)$$

\item Computation of $\mult2$. \\
Assume that $\alpha_2=2$, $\alpha_3=\alpha_5=1$ and $\alpha_p=0$ for $p\in\mathcal P'\setminus\{2,3,5\}$. Let $n$ be such that $$\left\{\begin{array}{ll}n\equiv p^{\alpha_p}\mod p^{\alpha_p+1} & \text{ if }\alpha_p>0\\n\equiv 3 \mod p&\text{ if $\alpha_p=0$}\end{array}\right.$$

Then $n=60n'$ where no term in $\mathcal P'$ divides $n'$. No term in $Q$ divide $n-1$ and $n+1$. And $2$ but no other term in $Q$ divides $n-2$ and $n+2$. Thus
$$a_{n-2}=-\alpha\mathbb L^{-\sum_{i=1}^d\left\lfloor\frac{n-2}{k_i}\right\rfloor}\quad\text{ and }\quad a_{n+2}=-\alpha\mathbb L^{-\sum_{i=1}^d\left\lfloor\frac{n+2}{k_i}\right\rfloor}$$
where $\alpha$ is of the form $\bigast_{i=1}^{\mult2}\left(\mathbb1-[\varepsilon_i x^2:\mathbb R^*\rightarrow\mathbb R^*]\right)$. Thus
$$\beta(F^+(a_{n-2}))=\beta\left(-\mathbb L_{\AS}^{-\sum_{i=1}^d\left\lfloor\frac{n-2}{k_i}\right\rfloor}F^+(\alpha)\right)=-u^{-\sum_{i=1}^d\left\lfloor\frac{n-2}{k_i}\right\rfloor}\beta\left(F^+(\alpha)\right)$$
and similarly
$$\beta(F^+(a_{n+2}))=\beta\left(-\mathbb L_{\AS}^{-\sum_{i=1}^d\left\lfloor\frac{n+2}{k_i}\right\rfloor}F^+(\alpha)\right)=-u^{-\sum_{i=1}^d\left\lfloor\frac{n+2}{k_i}\right\rfloor}\beta\left(F^+(\alpha)\right)$$
We deduce from Lemma \ref{lem:TSchiFplus} and $\chi_c\left(F^+\left(\mathbb 1-[\pm x^2]\right)\right)=\mp 1$ that $$\beta\left(F^+(\alpha)\right)(u=-1)=\chi_c(F^+(\alpha))=\prod_{i=1}^{\mult2}\chi_c(F^+(\mathbb 1-[\varepsilon_ix^2]))=\pm1$$
Thus $\beta\left(F^+(\alpha)\right)\neq0$ and 
$$-\deg\beta(F^+(a_{n+2}))+\deg\beta(F^+(a_{n-2}))=\mult2+\sum_{q\in Q,q|60}\mult(q)$$
Notice that in the first case we got an equation of the form $$cst=\sum_{q\in Q,q|60}\mult(q)$$
Therefore $\mult2$ may be derived from $Z_f(T)$.

\item Computation of $\mult q$ with $2\nmid q$ and $3\nmid q$. \\
Assume that $\alpha_2=\alpha_3=0$ and that $0\le\alpha_p\le\gamma_p$ for each $p\in\mathcal P'\setminus\{2,3\}$. Let $n$ be such that $$\left\{\begin{array}{ll}n\equiv p^{\alpha_p}\mod p^{\alpha_p+1} & \text{ if }\alpha_p>0\\n\equiv 1 \mod 8& \\ n\equiv 1\mod p&\text{ if }p\neq2,\,\alpha_p=0\end{array}\right.$$
Then the only elements of $Q$ which divide $n$ are those of the form $\displaystyle\prod_{p\in\mathcal P'}p^{\beta_p}$ with $0\le\beta_p\le\alpha_p$, the only element in $Q$ which divides $n+1$ and $n-3$ is $2$, no element in $Q$ divides $n-2$ and $6$ divides $n-1$.
Thus $$-\deg\beta(\overline{a_{n+1}})+\deg\beta(\overline{a_{n-3}})=\sum_{q\in Q,q|n-1}\mult(q)+\sum_{0\le\beta_p\le\alpha_p}\mult\left(\prod_{p\in\mathcal P'}p^{\beta_p}\right)+\mult(2)$$
Notice that since in the first case we got an equation of the form $$cst=\sum_{q\in Q,q|n-1}\mult(q)$$
and that we already know $\mult2$, thus we get an equation of the form
$$cst=\sum_{0\le\beta_p\le\alpha_p}\mult\left(\prod_{p\in\mathcal P'}p^{\beta_p}\right)$$
Remark: either $5|n$ (if $5|q$) or $5|n-1$.

Therefore we may recursively compute $\mult q$ for each $q\in Q$ such that $2\nmid q$ and $3\nmid q$ by varying $\alpha_p$ for each $p\in\mathcal P'\setminus\{2,3\}$.

\item Computation of $\mult3$ and $\mult4$. Let $n$ be such that $$\left\{\begin{array}{ll}n\equiv4\mod8&\\n\equiv4\mod9&\\n\equiv4\mod25&\\n\equiv5\mod7&\\n\equiv4\mod p&\text{ if $p\in\mathcal P'\setminus\{2,3,5,7\}$}\end{array}\right.$$ Then $2,4$ are the only element of $Q$ dividing $n$, $3$ is the only element of $Q$ dividing $n-1$, $5$ is the only element of $Q$ dividing $n+1$, $2$ is the only element of $Q$ dividing $n-2$, $6$ divides $n+2$ and no element of $Q$ divides $n-3,n+3$. Thus $$-\deg\beta(\overline{a_{n+3}})+\deg\beta(\overline{a_{n-3}})=\sum_{q\in Q,q|n+2}\mult(q)+2\mult2+\mult3+\mult4+\mult5$$
Notice that we already know $\mult2$, $\mult5$ and that in the previous case we got an equation of the form $$c=\sum_{q\in Q,q|n+2}\mult(q)$$
So we have an equation of the form 
\begin{equation}\label{eqn:34-1}
\mult3+\mult4=cst
\end{equation}

Now let $\alpha_2=3,\alpha_3=2,\alpha_5=1,\alpha_7=1$. Let $n$ be such that $$\left\{\begin{array}{ll}n\equiv p^{\alpha_p}\mod p^{\alpha_p+1} & \text{ if }\alpha_p>0\\n\equiv5\mod p&\text{ if $\alpha_p=0$}\end{array}\right.$$
Then $2^3\cdot3^2\cdot5\cdot7$ divides $n$, no term of $Q$ divides $n-1,n+1$, only $2$ divides $n-2,n+2$, only $3$ divides $n-3,n+3$ and only $2,4$ divide $n-4,n+4$.

Since $\chi_c\left(F^+\left(\mathbb 1-[\pm x^2]\right)\right)=\pm 1$ and $\chi_c\left(F^+\left(\mathbb 1-[\pm x^4]\right)\right)=\pm 1$, we have $$-\deg\beta(F^+(a_{n+4}))+\deg\beta(F^+(a_{n-4}))=3\mult2+2\mult3+\mult4+\sum_{q\in Q,q|n}\mult(q)$$ but in the first case we got an equation of the form $\sum_{q\in Q,q|n}\mult(q)=cst$ and we already know $\mult2$, thus we get an equation of the form 
\begin{equation}\label{eqn:34-2}
2\mult3+\mult4=cst
\end{equation}

We compute $\mult3$ and $\mult4$ from the system given by (\ref{eqn:34-1}) and (\ref{eqn:34-2}).

\item Computation of $\mult q$ with $2|q$, $4\nmid q$, $3\nmid q$ and $5\nmid q$. \\
Let $\alpha_2=1$, $\alpha_3=\alpha_5=0$ and $0\le\alpha_p\le\gamma_p$ otherwise. Let $n$ be such that $$\left\{\begin{array}{ll}n\equiv p^{\alpha_p}\mod p^{\alpha_p+1} & \text{ if $\alpha_p>0,p\neq2$}\\n\equiv2\mod8\\n\equiv2\mod9&\\n\equiv6\mod25&\\n\equiv2\mod7&\text{ if $\alpha_7=0$}\\n\equiv 4\mod p&\text{ if $\alpha_p=0$ and $p\neq3,5,7$}\end{array}\right.$$ Then the only terms of $Q$ dividing $n$ are the divisors of $\prod p^{\alpha_p}$, the only term of $Q$ dividing $n-1$ is $5$, the only term of $Q$ dividing $n+1$ is $3$, the only terms of $Q$ dividing $n+2$ are $2,4$, no term of $Q$ divides $n-3,n+3$ and $6$ divides $n-2$. Thus
$$-\deg\beta(\overline{a_{n+3}})+\deg\beta(\overline{a_{n-3}})=\mult2+\mult3+\mult4+\mult5+\sum_{q\in Q,q|n-2}\mult(q)+\sum_{0\le\beta_p\le\alpha_p}\mult\left(\prod_{p\in\mathcal P'}p^{\beta_p}\right)$$
from which we derive an equation of the form 
$$\sum_{\substack{\beta_2=1\\0\le\beta_p\le\alpha_p}}\mult\left(\prod_{p\in\mathcal P'}p^{\beta_p}\right)=cst$$
since we already know $\mult q$ for $q\in Q$ with $2\nmid q$ and $3\nmid q$ and since we got in the first case an equation of the form $$\sum_{q\in Q,q|n-2}\mult(q)=cst$$

Therefore we may recursively compute $\mult q$ for each $q\in Q$ such that $2|q$, $4\nmid q$, $3\nmid q$ and $5\nmid q$ by varying the $\alpha_p$.

\item Computation of $\mult q$ with $4|q$, $3\nmid q$ and $5\nmid q$. \\
Let $2\le\alpha_2\le\gamma_2$, $\alpha_3=\alpha_5=0$ and $0\le\alpha_p\le\gamma_p$ otherwise. Let $n$ be such that $$\left\{\begin{array}{ll}n\equiv p^{\alpha_p}\mod p^{\alpha_p+1} & \text{ if $\alpha_p>0$}\\n\equiv2\mod9\\n\equiv6\mod25&\\n\equiv2\mod7&\text{if $\alpha_7=0$}\\n\equiv 4 \mod p&\text{if $\alpha_p=0,p\neq3,5,7$}\end{array}\right.$$ Then the only terms of $Q$ dividing $n$ are the divisors of $\prod p^{\alpha_p}$, the only term of $Q$ dividing $n-1$ is $5$, the only term of $Q$ dividing $n+1$ is $3$, the only term of $Q$ dividing $n+2$ is $2$, $6$ divides $n-2$ and no term of $Q$ divides $n-3,n+3$. Thus
$$-\deg\beta(\overline{a_{n+3}})+\deg\beta(\overline{a_{n-3}})=\mult2+\mult3+\mult5+\sum_{q\in Q,q|n-2}\mult(q)+\sum_{0\le\beta_p\le\alpha_p}\mult\left(\prod_{p\in\mathcal P'}p^{\beta_p}\right)$$
from which we derive an equation of the form
$$\sum_{\substack{2\le\beta_2\le\alpha_2\\0\le\beta_p\le\alpha_p}}\mult\left(\prod_{p\in\mathcal P'}p^{\beta_p}\right)=cst$$
Therefore we may recursively compute $\mult q$ for each $q\in Q$ such that $4|q$, $3\nmid q$ and $5\nmid q$ by varying the $\alpha_p$.

\item Computation of $\mult q$ with $2|q$, $4\nmid q$, $3\nmid q$, $5|q$. Let $\alpha_2=1$, $\alpha_3=0$, $1\le\alpha_5\le\gamma_5$ and $0\le\alpha_p\le\gamma_p$ otherwise. Let $n$ be such that $$\left\{\begin{array}{ll}n\equiv p^{\alpha_p}\mod p^{\alpha_p+1} & \text{ if }\alpha_p>0\text{ and }p\neq2\\ n\equiv 2\mod8&\\n\equiv2\mod9\\n\equiv 3\mod p&\text{ if }\alpha_p=0\text{ and }p\neq3\end{array}\right.$$ Then the only terms of $Q$ dividing $n$ are the divisors of $\prod p^{\alpha_p}$, the only term in $Q$ dividing $n+1$ is $3$ and the only terms in $Q$ dividing $n+2$ are $2,4$. No term of $Q$ divides $n-1,n+3$. Thus $$-\deg\beta(\overline{a_{n+3}})+\deg\beta(\overline{a_{n-1}})=\mult3+\mult2+\mult4+\sum_{0\le\beta_p\le\alpha_p}\mult\left(\prod_{p\in\mathcal P'}p^{\beta_p}\right)$$

We conclude as in the previous cases.

\item Computation of $\mult q$ with $4|q$, $3\nmid q$, $5|q$. Let $2\le\alpha_2\le\gamma_2$, $\alpha_3=0$, $1\le\alpha_5\le\gamma_5$ and $0\le\alpha_p\le\gamma_p$ otherwise. Let $n$ be such that $$\left\{\begin{array}{ll}n\equiv p^{\alpha_p}\mod p^{\alpha_p+1} & \text{ if }\alpha_p>0\\n\equiv2\mod9&\\n\equiv 3\mod p&\text{ if }\alpha_p=0\text{ and }p\neq3\end{array}\right.$$ Then the only terms of $Q$ dividing $n$ are the divisors of $\prod p^{\alpha_p}$, the only term in $Q$ dividing $n+1$ is $3$ and the only terms in $Q$ dividing $n+2$ are $2$. No term of $Q$ divides $n-1,n+3$. Thus $$-\deg\beta(\overline{a_{n+3}})+\deg\beta(\overline{a_{n-1}})=\mult3+\mult2+\sum_{0\le\beta_p\le\alpha_p}\mult\left(\prod_{p\in\mathcal P'}p^{\beta_p}\right)$$

We conclude as in the previous cases.

\item Computation of $\mult q$ with $3|q$, $2\nmid q$ and $5\nmid q$. Let $1\le\alpha_3\le\gamma_3$, $\alpha_2,\alpha_5=0$ and $0\le\alpha_p\le\gamma_p$ otherwise. Let $n$ be such that $$\left\{\begin{array}{ll}n\equiv p^{\alpha_p}\mod p^{\alpha_p+1} & \text{ if }\alpha_p>0\\ n\equiv 3\mod8&\\n\equiv4\mod25\\ n\equiv 3\mod p&\text{ is $\alpha_p=0$, $p\neq2$, $p\neq5$}\end{array}\right.$$ Then the only terms of $Q$ dividing $n$ are the divisors of $\prod p^{\alpha_p}$, 2 is the only term of $Q$ dividing $n-1$, $2,4,5,10,20$ are the only terms of $Q$ dividing $n+1$ and no term in $Q$ divides $n-2,n+2$. Thus $$-\deg\beta(\overline{a_{n+2}})+\deg\beta(\overline{a_{n-2}})=\mult2+\sum_{q\in Q,q|20}\mult(q)+\sum_{0\le\beta_p\le\alpha_p}\mult\left(\prod_{p\in\mathcal P'}p^{\beta_p}\right)$$
Notice that in the previous case we got an equation of the form $$c=\sum_{q\in Q,q|20}\mult(q)$$

We conclude as in the previous cases.

\item Computation of $\mult q$ with $3|q$, $2\nmid q$ and $5|q$. Let $1\le\alpha_3\le\gamma_3$, $\alpha_2=0$, $1\le\alpha_5\le\gamma_5$ and $0\le\alpha_p\le\gamma_p$ otherwise. Let $n$ be such that $$\left\{\begin{array}{ll}n\equiv p^{\alpha_p}\mod p^{\alpha_p+1} & \text{ if }\alpha_p>0\\ n\equiv 3\mod8 \\ n\equiv 3\mod p&\text{ if $\alpha_p=0$, $p\neq2$}\end{array}\right.$$ Then the only terms of $Q$ dividing $n$ are the divisors of $\prod p^{\alpha_p}$, 2 is the only term of $Q$ dividing $n-1$, $2,4$ are the only terms of $Q$ dividing $n+1$ and no term in $Q$ divides $n-2,n+2$. Thus $$-\deg\beta(\overline{a_{n+2}})+\deg\beta(\overline{a_{n-2}})=2\mult2+\mult4+\sum_{0\le\beta_p\le\alpha_p}\mult\left(\prod_{p\in\mathcal P'}p^{\beta_p}\right)$$

We conclude as in the previous cases.

\item Computation of $\mult q$ for $q\in Q$ satisfying $6|q$. \\
Let $\alpha_2,\alpha_3$ be such that $1\le\alpha_2\le\gamma_2$ and $1\le\alpha_3\le\gamma_3$ and for each $p\in\mathcal P'\setminus\{2,3\}$, let $\alpha_p$ be such that $0\le\alpha_p\le\gamma_p$. \\
Now that we know $\mult q$ for $q\in Q$ satisfying $2\nmid q$ or $3\nmid q$ we may deduce from the equation of the first case an equation of the form
$$\sum_{\substack{1\le\beta_2\le\alpha_2\\1\le\beta_3\le\alpha_3\\0\le\beta_p\le\alpha_p}}\mult\left(\prod_{p\in\mathcal P'}p^{\beta_p}\right)=cst$$
And we conclude as in the previous cases by varying the $\alpha_p$.
\end{itemize}
\end{proof}


\footnotesize
\begin{thebibliography}{10}

\bibitem{AKMW}
{\sc D.~Abramovich, K.~Karu, K.~Matsuki, and J.~W{\l}odarczyk},
  \href{http://dx.doi.org/10.1090/S0894-0347-02-00396-X}{{\em Torification and
  factorization of birational maps}}, J. Amer. Math. Soc., 15 (2002),
  pp.~531--572 (electronic).

\bibitem{AGV-T2}
{\sc V.~I. Arnold, S.~M. Gusein-Zade, and A.~N. Varchenko}, {\em Singularities
  of differentiable maps. {V}olume 2}, Modern Birkh\"auser Classics,
  Birkh\"auser/Springer, New York, 2012.
\newblock Monodromy and asymptotics of integrals, Translated from the Russian
  by Hugh Porteous and revised by the authors and James Montaldi, Reprint of
  the 1988 translation.

\bibitem{AM65}
{\sc M.~Artin and B.~Mazur}, {\em On periodic points}, Ann. of Math. (2), 81
  (1965), pp.~82--99.

\bibitem{BM90}
{\sc E.~Bierstone and P.~D. Milman},
  \href{http://dx.doi.org/10.1007/BF01231509}{{\em Arc-analytic functions}},
  Invent. Math., 101 (1990), pp.~411--424.

\bibitem{BM97}
\leavevmode\vrule height 2pt depth -1.6pt width 23pt,
  \href{http://dx.doi.org/10.1007/s002220050141}{{\em Canonical
  desingularization in characteristic zero by blowing up the maximum strata of
  a local invariant}}, Invent. Math., 128 (1997), pp.~207--302.

\bibitem{BCR}
{\sc J.~Bochnak, M.~Coste, and M.-F. Roy},
  \href{http://dx.doi.org/10.1007/978-3-662-03718-8}{{\em Real algebraic
  geometry}}, vol.~36 of Ergebnisse der Mathematik und ihrer Grenzgebiete (3)
  [Results in Mathematics and Related Areas (3)], Springer-Verlag, Berlin,
  1998.
\newblock Translated from the 1987 French original, Revised by the authors.

\bibitem{jbc1}
{\sc J.-B. Campesato}, \href{http://arxiv.org/abs/arXiv:1406.6637}{{\em An
  inverse mapping theorem for blow-{N}ash maps on singular spaces}}, 2014,
  arXiv:1406.6637.
\newblock To appear in Nagoya Math. J.

\bibitem{CF14}
{\sc G.~Comte and G.~Fichou},
  \href{http://dx.doi.org/10.2140/gt.2014.18.963}{{\em Grothendieck ring of
  semialgebraic formulas and motivic real {M}ilnor fibers}}, Geom. Topol., 18
  (2014), pp.~963--996.

\bibitem{Den87}
{\sc J.~Denef}, \href{http://dx.doi.org/10.2307/2374583}{{\em On the degree of
  {I}gusa's local zeta function}}, Amer. J. Math., 109 (1987), pp.~991--1008.

\bibitem{DH01}
{\sc J.~Denef and K.~Hoornaert},
  \href{http://dx.doi.org/10.1006/jnth.2000.2606}{{\em Newton polyhedra and
  {I}gusa's local zeta function}}, J. Number Theory, 89 (2001), pp.~31--64.

\bibitem{DL92}
{\sc J.~Denef and F.~Loeser}, \href{http://dx.doi.org/10.2307/2152708}{{\em
  Caract\'eristiques d'{E}uler-{P}oincar\'e, fonctions z\^eta locales et
  modifications analytiques}}, J. Amer. Math. Soc., 5 (1992), pp.~705--720.

\bibitem{DL98}
\leavevmode\vrule height 2pt depth -1.6pt width 23pt, {\em Motivic {I}gusa zeta
  functions}, J. Algebraic Geom., 7 (1998), pp.~505--537.

\bibitem{DL99-Germs}
\leavevmode\vrule height 2pt depth -1.6pt width 23pt,
  \href{http://dx.doi.org/10.1007/s002220050284}{{\em Germs of arcs on singular
  algebraic varieties and motivic integration}}, Invent. Math., 135 (1999),
  pp.~201--232.

\bibitem{DL99-TS}
\leavevmode\vrule height 2pt depth -1.6pt width 23pt,
  \href{http://dx.doi.org/10.1215/S0012-7094-99-09910-6}{{\em Motivic
  exponential integrals and a motivic {T}hom-{S}ebastiani theorem}}, Duke Math.
  J., 99 (1999), pp.~285--309.

\bibitem{DL01}
\leavevmode\vrule height 2pt depth -1.6pt width 23pt, {\em Geometry on arc
  spaces of algebraic varieties}, in European {C}ongress of {M}athematics,
  {V}ol. {I} ({B}arcelona, 2000), vol.~201 of Progr. Math., Birkh\"auser,
  Basel, 2001, pp.~327--348.

\bibitem{DL02}
\leavevmode\vrule height 2pt depth -1.6pt width 23pt,
  \href{http://dx.doi.org/10.1016/S0040-9383(01)00016-7}{{\em Lefschetz numbers
  of iterates of the monodromy and truncated arcs}}, Topology, 41 (2002),
  pp.~1031--1040.

\bibitem{Fic05}
{\sc G.~Fichou}, \href{http://dx.doi.org/10.1112/S0010437X05001168}{{\em
  Motivic invariants of arc-symmetric sets and blow-{N}ash equivalence}},
  Compos. Math., 141 (2005), pp.~655--688.

\bibitem{Fic05-bis}
\leavevmode\vrule height 2pt depth -1.6pt width 23pt,
  \href{http://dx.doi.org/10.4064/ap87-0-10}{{\em Zeta functions and
  blow-{N}ash equivalence}}, Ann. Polon. Math., 87 (2005), pp.~111--126.

\bibitem{Fic06}
\leavevmode\vrule height 2pt depth -1.6pt width 23pt,
  \href{http://dx.doi.org/10.2996/kmj/1143122384}{{\em The corank and the index
  are blow-{N}ash invariants}}, Kodai Math. J., 29 (2006), pp.~31--40.

\bibitem{Fic08}
\leavevmode\vrule height 2pt depth -1.6pt width 23pt,
  \href{http://projecteuclid.org/euclid.jmsj/1212156658}{{\em Blow-{N}ash types
  of simple singularities}}, J. Math. Soc. Japan, 60 (2008), pp.~445--470.

\bibitem{FF}
{\sc G.~Fichou and T.~Fukui}, {\em Motivic invariants of real polynomial
  functions and their {N}ewton polyhedrons}, Math. Proc. Cambridge Philos.
  Soc., 160 (2016), pp.~141--166.

\bibitem{Gui02}
{\sc G.~Guibert}, \href{http://dx.doi.org/10.1007/PL00012442}{{\em Espaces
  d'arcs et invariants d'{A}lexander}}, Comment. Math. Helv., 77 (2002),
  pp.~783--820.

\bibitem{GLM}
{\sc G.~Guibert, F.~Loeser, and M.~Merle},
  \href{http://dx.doi.org/10.1215/S0012-7094-06-13232-5}{{\em Iterated
  vanishing cycles, convolution, and a motivic analogue of a conjecture of
  {S}teenbrink}}, Duke Math. J., 132 (2006), pp.~409--457.

\bibitem{GNA}
{\sc F.~Guill{\'e}n and V.~Navarro~Aznar},
  \href{http://dx.doi.org/10.1007/s102400200003}{{\em Un crit\`ere d'extension
  des foncteurs d\'efinis sur les sch\'emas lisses}}, Publ. Math. Inst. Hautes
  \'Etudes Sci.,  (2002), pp.~1--91.

\bibitem{KP03}
{\sc S.~Koike and A.~Parusi{\'n}ski},
  \href{http://aif.cedram.org/item?id=AIF_2003__53_7_2061_0}{{\em Motivic-type
  invariants of blow-analytic equivalence}}, Ann. Inst. Fourier (Grenoble), 53
  (2003), pp.~2061--2104.

\bibitem{Kou76}
{\sc A.~G. Kouchnirenko}, {\em Poly\`edres de {N}ewton et nombres de {M}ilnor},
  Invent. Math., 32 (1976), pp.~1--31.

\bibitem{Kuo79}
{\sc T.~C. Kuo}, {\em Une classification des singularit\'es r\'eelles}, C. R.
  Acad. Sci. Paris S\'er. A-B, 288 (1979), pp.~A809--A812.

\bibitem{Kuo85}
\leavevmode\vrule height 2pt depth -1.6pt width 23pt,
  \href{http://dx.doi.org/10.1007/BF01388802}{{\em On classification of real
  singularities}}, Invent. Math., 82 (1985), pp.~257--262.

\bibitem{Kur88}
{\sc K.~Kurdyka}, \href{http://dx.doi.org/10.1007/BF01460044}{{\em Ensembles
  semi-alg\'ebriques sym\'etriques par arcs}}, Math. Ann., 282 (1988),
  pp.~445--462.

\bibitem{Loo02}
{\sc E.~Looijenga}, {\em Motivic measures}, Ast\'erisque,  (2002),
  pp.~267--297.
\newblock S{\'e}minaire Bourbaki, Vol. 1999/2000.

\bibitem{MP03}
{\sc C.~McCrory and A.~Parusi{\'n}ski},
  \href{http://dx.doi.org/10.1016/S1631-073X(03)00168-7}{{\em Virtual {B}etti
  numbers of real algebraic varieties}}, C. R. Math. Acad. Sci. Paris, 336
  (2003), pp.~763--768.

\bibitem{MP11}
\leavevmode\vrule height 2pt depth -1.6pt width 23pt, {\em The weight
  filtration for real algebraic varieties}, in Topology of stratified spaces,
  vol.~58 of Math. Sci. Res. Inst. Publ., Cambridge Univ. Press, Cambridge,
  2011, pp.~121--160.

\bibitem{Nas52}
{\sc J.~F. Nash, Jr.}, {\em Real algebraic manifolds}, Ann. of Math. (2), 56
  (1952), pp.~405--421.

\bibitem{Nas95}
\leavevmode\vrule height 2pt depth -1.6pt width 23pt,
  \href{http://dx.doi.org/10.1215/S0012-7094-95-08103-4}{{\em Arc structure of
  singularities}}, Duke Math. J., 81 (1995), pp.~31--38 (1996).
\newblock A celebration of John F. Nash, Jr.

\bibitem{Par04}
{\sc A.~Parusi{\'n}ski},
  \href{http://dx.doi.org/10.1007/s00208-003-0486-x}{{\em Topology of injective
  endomorphisms of real algebraic sets}}, Math. Ann., 328 (2004), pp.~353--372.

\bibitem{Qua01}
{\sc R.~Quarez}, \href{http://aif.cedram.org/item?id=AIF_2001__51_1_43_0}{{\em
  Espace des germes d'arcs r\'eels et s\'erie de {P}oincar\'e d'un ensemble
  semi-alg\'ebrique}}, Ann. Inst. Fourier (Grenoble), 51 (2001), pp.~43--68.

\bibitem{Rai12}
{\sc M.~Raibaut}, {\em Singularit\'es \`a l'infini et int\'egration motivique},
  Bull. Soc. Math. France, 140 (2012), pp.~51--100.

\bibitem{Shi87}
{\sc M.~Shiota}, {\em Nash manifolds}, vol.~1269 of Lecture Notes in
  Mathematics, Springer-Verlag, Berlin, 1987.

\bibitem{Ste99}
{\sc N.~Steenrod}, {\em The topology of fibre bundles}, Princeton Landmarks in
  Mathematics, Princeton University Press, Princeton, NJ, 1999.
\newblock Reprint of the 1957 edition, Princeton Paperbacks.

\bibitem{Var76}
{\sc A.~N. Varchenko}, {\em Zeta-function of monodromy and {N}ewton's diagram},
  Invent. Math., 37 (1976), pp.~253--262.

\bibitem{Whi65}
{\sc H.~Whitney}, {\em Local properties of analytic varieties}, in Differential
  and {C}ombinatorial {T}opology ({A} {S}ymposium in {H}onor of {M}arston
  {M}orse), Princeton Univ. Press, Princeton, N. J., 1965, pp.~205--244.

\end{thebibliography}
\end{document}